\documentclass[11pt,a4paper,reqno]{amsart}
\usepackage[english]{babel}
\usepackage[T1]{fontenc}
\usepackage{verbatim}
\usepackage{palatino}
\usepackage{amsmath}
\usepackage{mathabx}
\usepackage{amssymb}
\usepackage{amsthm}
\usepackage{amsfonts}
\usepackage{graphicx}
\usepackage{esint}
\usepackage{xcolor}
\usepackage{mathtools}
\usepackage{overpic}
\usepackage{enumerate}

\usepackage[colorlinks = true, linkcolor={red!80!black}, citecolor = {blue!60!black}]{hyperref}
\author{Katrin F\"assler and Tuomas Orponen}
\title[Vertical projections]{Vertical projections in the Heisenberg group \\ via cinematic functions and point-plate incidences}
\address{Department of Mathematics and Statistics\\ University of Jyv\"askyl\"a,
P.O. Box 35 (MaD)\\
FI-40014 University of Jyv\"askyl\"a\\
Finland}  \email{katrin.s.fassler@jyu.fi}
\email{tuomas.t.orponen@jyu.fi}

\date{\today}
\thanks{K.F.\ is supported by the Academy of Finland via the project
\emph{Singular integrals, harmonic functions, and boundary
regularity in Heisenberg groups}, grant No.\ 321696. T.O.\ is
 supported by the Academy of Finland via the project \emph{Incidences on Fractals}, grant No.\ 321896. }

\newcommand{\R}{\mathbb{R}}
\newcommand{\W}{\mathbb{W}}
\newcommand{\He}{\mathbb{H}}
\newcommand{\N}{\mathbb{N}}

\newcommand{\Z}{\mathbb{Z}}
\newcommand{\tn}{\mathbb{P}}

\newcommand{\spt}{\operatorname{spt}}
\newcommand{\Hd}{\dim_{\mathrm{\mathbb{H}}}}

\newcommand{\E}{\mathbb{E}}

\newcommand{\spa}{\operatorname{span}}

\newcommand{\dist}{\operatorname{dist}}

\def\Barint_#1{\mathchoice
          {\mathop{\vrule width 6pt height 3 pt depth -2.5pt
                  \kern -8pt \intop}\nolimits_{#1}}%
          {\mathop{\vrule width 5pt height 3 pt depth -2.6pt
                  \kern -6pt \intop}\nolimits_{#1}}%
          {\mathop{\vrule width 5pt height 3 pt depth -2.6pt
                  \kern -6pt \intop}\nolimits_{#1}}%
          {\mathop{\vrule width 5pt height 3 pt depth -2.6pt
                  \kern -6pt \intop}\nolimits_{#1}}}

\numberwithin{equation}{section}

\theoremstyle{plain}
\newtheorem{thm}[equation]{Theorem}
\newtheorem*{"thm"}{"Theorem"}

\newtheorem{lemma}[equation]{Lemma}

\newtheorem{cor}[equation]{Corollary}
\newtheorem{proposition}[equation]{Proposition}

\theoremstyle{definition}

\newtheorem{definition}[equation]{Definition}

\theoremstyle{remark}
\newtheorem{remark}[equation]{Remark}

\newcommand{\nref}[1]{(\hyperref[#1]{#1})}

\DeclareMathSymbol{\intop}  {\mathop}{mathx}{"B3}

\begin{document}

\begin{abstract} Let $\{\pi_{e} \colon \mathbb{H} \to \mathbb{W}_{e} : e \in S^{1}\}$ be the family of vertical projections in the first Heisenberg group $\mathbb{H}$. We prove that if $K \subset \mathbb{H}$ is a Borel set with Hausdorff dimension $\dim_{\mathbb{H}} K \in [0,2] \cup \{3\}$, then 
$$ \dim_{\mathbb{H}} \pi_{e}(K) \geq \dim_{\mathbb{H}} K $$ 
for $\mathcal{H}^{1}$ almost every $e \in S^{1}$. This was known earlier if $\dim_{\mathbb{H}} K \in [0,1]$.

The proofs for $\dim_{\mathbb{H}} K \in [0,2]$ and $\dim_{\mathbb{H}} K = 3$ are based on different techniques. For $\dim_{\mathbb{H}} K \in [0,2]$, we reduce matters to a Euclidean problem, and apply the method of cinematic functions due to Pramanik, Yang, and Zahl. 

To handle the case $\dim_{\mathbb{H}} K = 3$, we introduce a point-line duality between horizontal lines and conical lines in $\mathbb{R}^{3}$. This allows us to transform the Heisenberg problem into a point-plate incidence question in $\R^{3}$. To solve the latter, we apply a Kakeya inequality for plates in $\R^{3}$, due to Guth, Wang, and Zhang. This method also yields partial results for Borel sets $K \subset \mathbb{H}$ with $\dim_{\mathbb{H}} K \in (5/2,3)$.
\end{abstract}

\maketitle

\tableofcontents

\section{Introduction}

Fix $e \in S^{1} \times \{0\} \subset \He$, and consider the \emph{vertical plane} $\mathbb{W}_{e} := e^{\perp}$ in the first Heisenberg group $\He$, see Section \ref{s:heis} for the definitions. Every point $p \in \He$ can be uniquely decomposed as $p = w \cdot v$, where
\begin{displaymath} w \in \mathbb{W}_{e} \quad \text{and} \quad v \in \mathbb{L}_{e} := \spa(e). \end{displaymath}
This decomposition gives rise to the \emph{vertical projection} $\pi_{e} := \pi_{\mathbb{W}_{e}} \colon \He \to \mathbb{W}_{e}$, defined by $\pi_{e}(p) := w$. A good way to visualise $\pi_{e}$ is to note that the fibres $\pi_{e}^{-1}\{w\}$, $w \in \W_{e}$, coincide with the horizontal lines $w \cdot \mathbb{L}_{e}$. These lines foliate $\mathbb{H}$, as $w$ ranges in $\mathbb{W}_{e}$, but are not parallel. Thus, the projections $\pi_{e}$ are non-linear maps with linear fibres. For example, in the special cases $e_{1} = (1,0,0)$ and $e_{2} = (0,1,0)$ we have the concrete formulae
\begin{equation}\label{projections} \pi_{e_{1}}(x,y,t) = \left( 0,y,t + \tfrac{xy}{2} \right) \quad \text{and} \quad \pi_{e_{2}}(x,y,t) = \left( x,0,t - \tfrac{xy}{2} \right). \end{equation}
From the point of view of geometric measure theory in the Heisenberg group, the vertical projections are the Heisenberg analogues of orthogonal projections to $(d - 1)$-planes in $\R^{d}$. One of the fundamental theorems concerning orthogonal projections in $\R^{d}$ is the \emph{Marstrand-Mattila projection theorem} \cite{Mar,MR0409774}: if $K \subset \R^{d}$ is a Borel set, then
\begin{equation}\label{mm} \dim_{\mathrm{E}} \pi_{V}(K) = \min\{\dim_{\mathrm{E}} K,d - 1\} \end{equation}
for almost all $(d - 1)$-planes $V \subset \mathbb{R}^{d}$. Here $\dim_{\mathrm{E}}$ refers to Hausdorff dimension in Euclidean space -- in contrast to the notation "$\Hd$" which will refer to Hausdorff dimension in the Heisenberg group. In $\R^{d}$, orthogonal projections are Lipschitz maps, so the upper bound in \eqref{mm} is trivial, and the main interest in \eqref{mm} is the lower bound.

The vertical projections $\pi_{e}$ are not Lipschitz maps $\He \to
\W_{e}$ relative to the natural metric $d_{\He}$ in $\He$ and
$\W_{e}$. Indeed, they can increase Hausdorff dimension: an easy
example is a horizontal line, which is $1$-dimensional to begin
with, but gets projected to a $2$-dimensional set -- a parabola --
in almost all directions. For general (sharp) results on how much
$\pi_{e}$ can increase Hausdorff dimension, see \cite[Theorem
1.3]{MR3047423}. We note that the vertical planes $\mathbb{W}_{e}$
themselves are $3$-dimensional, and $\He$ is $4$-dimensional.

Can the vertical projections lower Hausdorff dimension? In some directions they can, and the general (sharp) universal lower bound was already found in \cite[Theorem 1.3]{MR3047423}:
\begin{displaymath} \Hd \pi_{e}(K) \geq \max\{0,\tfrac{1}{2}(\Hd K - 1),2\Hd K - 5\}, \qquad e \in S^{1}. \end{displaymath}
Our main result states that the dimension drop cannot occur in a set of directions of positive measure for sets of dimension in $[0,2] \cup \{3\}$:
\begin{thm}\label{mainSimple} Let $K \subset \He$ be a Borel set with $\Hd K \in [0,2] \cup \{3\}$. Then $\Hd \pi_{e}(K) \geq \Hd K$ for $\mathcal{H}^{1}$ almost every $e \in S^{1}$. \end{thm}

The result is sharp for all values $\Hd K \in [0,2] \cup \{3\}$, and new for $\Hd K \in (1,2] \cup \{3\}$. It makes progress in \cite[Conjecture 1.5]{MR3047423} which proposes that
\begin{equation}\label{mainConj} \Hd \pi_{e}(K) \geq \min\{\Hd K,3\} \end{equation}
for $\mathcal{H}^{1}$ almost every $e \in S^{1}$. The cases $\Hd K \in [0,1]$ were established around a decade ago by Balogh, Durand-Cartagena, the first author, Mattila, and Tyson \cite[Theorem 1.4]{MR3047423}. For $\Hd K > 1$, the strongest previous partial result is due to Harris \cite{Harris} who in 2022 proved that
\begin{displaymath} \Hd \pi_{e}(K) \geq \min\left\{\frac{1 + \Hd K}{2},2\right\} \quad \text{for $\mathcal{H}^{1}$ a.e. $e \in S^{1}$}. \end{displaymath}
Other partial results, also higher dimensions, are contained in \cite{MR2955184,MR3495435,MR4112255,MR4321543}.

The "disconnected" assumption $\Hd K \in [0,2] \cup \{3\}$ is due to the fact that Theorem \ref{mainSimple} is a combination of two separate results, with different proofs. Perhaps surprisingly, the cases $\Hd K \in [0,2]$ are a consequence of a "$1$-dimensional" projection theorem. Namely, consider the (nonlinear) projections $\rho_{e} \colon \R^{3} \to \R$ obtained as the $t$-coordinates of the projections $\pi_{e}$: 
\begin{equation}\label{form63} \rho_{e} = \pi_{T} \circ \pi_{e}, \qquad \pi_{T}(x,y,t) = (0,0,t). \end{equation}
Since the $t$-axis in $\He$ is $2$-dimensional, it is conceivable that the maps $\rho_{e}$ do not a.e. lower the Hausdorff dimension of Borel sets of dimension at most $2$. This is what we prove: 

\begin{thm}\label{mainLowDimIntro} Let $K \subset \R^{3}$ be a Borel set. Then 
\begin{displaymath} \dim_{\mathrm{E}} \rho_{e}(K) = \min\{\dim_{\mathrm{E}} K,1\} \quad \text{and} \quad \Hd \rho_{e}(K) \geq \min\{\Hd K,2\} \end{displaymath}
for $\mathcal{H}^{1}$ almost every $e \in S^{1}$. In fact, the following shaper conclusion holds: for $0 \leq s < \min\{\Hd K,2\}$, we have $\dim_{\mathrm{E}} \{e \in S^{1} : \Hd \rho_{e}(K) \leq s\} \leq \tfrac{s}{2}$. \end{thm}

Theorem \ref{mainLowDimIntro} implies the cases $\Hd K \in [0,2]$ of Theorem \ref{mainSimple}, because the map $\pi_{T}$ is Lipschitz when restricted to any plane $\W_{e}$, thus $\Hd \pi_{e}(K) \geq \Hd \rho_{e}(K)$ for all $e \in S^{1}$.

The proof of Theorem \ref{mainLowDimIntro} is a fairly straightforward application of recently developed technology to study the \emph{restricted projections} problem in $\R^{3}$ (see \cite{2022arXiv220713844G,2022arXiv220915152G,2022arXiv220806896H,2017arXiv170804859K,2022arXiv220702259P}). Even though the maps $\rho_{e}$ are nonlinear, Theorem \ref{mainLowDimIntro} falls within the scope of the \emph{cinematic function} framework introduced by Pramanik, Yang, and Zahl \cite{2022arXiv220702259P}. In Theorem \ref{mainTransversal}, we apply this framework to record a more general version of Theorem \ref{mainLowDimIntro} which simultaneously generalises \cite[Theorem 1.3]{2022arXiv220702259P} and Theorem \ref{mainLowDimIntro}. The details can be found in Section \ref{s:Dim2}.

The case $\Hd K = 3$ of Theorem \ref{mainSimple} is the harder result. This time we do not know how to deduce it from a purely Euclidean statement. Instead, it is deduced from the following "mixed" result between Heisenberg and Euclidean metrics:

\begin{thm}\label{main} Let $K \subset \He$ be a Borel set with $\Hd K \geq 2$. Then,
\begin{equation}\label{mainEuc} \dim_{\mathrm{E}} \pi_{e}(K) \geq \min\{\Hd K - 1,2\} \end{equation}
for $\mathcal{H}^{1}$ almost every $e \in S^{1}$, and consequently
\begin{equation}\label{mainHeis} \Hd \pi_{e}(K) \geq\min\{ 2\Hd K - 3 ,3\}\end{equation}
for $\mathcal{H}^{1}$ almost every $e \in S^{1}$.
\end{thm}

Theorem \ref{main} will further be deduced from a $\delta$-discretised result which may have independent interest. We state here a simplified version (the full version is Theorem \ref{mainDiscrete}):
\begin{thm}\label{mainMeasure} Let $0 \leq t \leq 3$ and $\eta > 0$. Then, the following holds for $\delta,\epsilon > 0$ small enough,
depending only on $\eta$. Let $\mathcal{B}$ be a non-empty
$(\delta,t,\delta^{-\epsilon})$-set of Heisenberg balls of radius
$\delta$, all contained in $B_{\He}(1)$. Then, there exists $e \in
S^{1}$ such that
\begin{equation}\label{form45} \mathrm{Leb}(\pi_{e}(\cup \mathcal{B})) \geq \delta^{3 - t + \eta}. \end{equation}
\end{thm}

Here $\mathrm{Leb}$ denotes Lebesgue measure on $\W_{e}$, identified with $\R^{2}$. For the definition of $(\delta,t)$-sets of $\delta$-balls, see Definition \ref{deltaTSet}. Theorems \ref{main} and \ref{mainMeasure} are proved in Sections \ref{s:discretisation}-\ref{s:mainProof}.

\begin{remark} It seems likely that the lower bound \eqref{form45} remains valid under the alternative assumptions that $|\mathcal{B}| = \delta^{-t}$ and 
\begin{equation}\label{form40} |\{B \in \mathcal{B} : B \subset B_{\He}(p,r)\}| \leq \delta^{-\epsilon} \cdot \left(\frac{r}{\delta} \right)^{3}, \qquad p \in \He, \, r \geq \delta. \end{equation}
This is because the estimate \eqref{form45} ultimately follows from Proposition \ref{prop5} which works under the non-concentration condition \eqref{form40}. We will not need this version of Theorem \ref{mainMeasure}, so we omit the details.  \end{remark}

\subsection{Sharpness of the results} Theorem \ref{mainSimple} is sharp for all values $\Hd K \in [0,2] \cup \{3\}$. The "mixed" inequality \eqref{mainEuc} in Theorem \ref{main} is sharp for all values $\Hd K \geq 2$, even though the Heisenberg corollary \eqref{mainHeis} is unlikely to be sharp for any value $\Hd K < 3$ (in fact, Theorem \ref{mainSimple} shows that \eqref{mainHeis} is \textbf{not} sharp for $\Hd K < 5/2)$. 

The sharpness examples are as follows: if $s := \Hd K \leq 2$, take an $s$-dimensional subset of the $t$-axis, and note that the $t$-axis is preserved by the projections $\rho_{e}$ and $\pi_{e}$. If $s > 2$, take $K$ to be a union of translates of the $t$-axis, thus $K := K_{0} \times \R$. The $\pi_{e}$-projections send vertical lines to vertical lines, so $\pi_{e}(K)$ is a union of vertical lines on $\W_{e}$; more precisely $\pi_{e}(K) = \bar{\pi}_{e}(K_{0}) \times \R$, where $\bar{\pi}_{e}$ is an orthogonal projection in $\R^{2}$. These observations lead to the sharpness of \eqref{mainEuc}, and the sharpness of conjecture \eqref{mainConj}.

Theorem \ref{mainMeasure} is sharp for all values of $t \in [0,3]$. Indeed, it is possible that $|\mathcal{B}| = \delta^{-t}$, and then $\mathrm{Leb}(\pi_{e}(\cup \mathcal{B})) \lesssim \delta^{3 - t}$ for every $e \in S^{1}$. It also follows from \eqref{form45} that the smallest number of $d_{\He}$-balls of radius $\delta$ needed to cover $\pi_{e}(\cup \mathcal{B})$ is $\gtrsim \delta^{-t + \eta}$. One might think that this solves Conjecture \ref{mainConj} for all $\Hd K \in [0,3]$, but we were not able to make this deduction rigorous: the difficulty appears when attempting to $\delta$-discretise Conjecture \ref{mainConj}, and is caused by the non-Lipschitz behaviour of $\pi_{e} \colon (\He,d_{\He}) \to (\W_{e},d_{\He})$. This problem will be apparent in the proof of Theorem \ref{main} in Section \ref{s:mainProof}. Another, more heuristic, way of understanding the difference between Theorem \ref{mainMeasure} and Conjecture \ref{mainConj} is this: $\mathrm{Leb}(\pi_{e}(K))$ is invariant under left-translating $K$, but $\Hd \pi_{e}(K)$ is generally not. 

As we already explained, the proof of Theorem \ref{mainLowDimIntro}, therefore the cases $\Hd K \in [0,2]$ of Theorem \ref{mainSimple}, follow from recent developments in the theory of \emph{restricted projections} in $\R^{3}$, notably the \emph{cinematic function} framework in \cite{2022arXiv220702259P}. The proof of Theorem \ref{main} does not directly overlap with these results (see Section \ref{s:outline} for more details), and for example does not use the $\ell^{2}$-decoupling theorem, in contrast with \cite{2022arXiv220713844G,2022arXiv220915152G,2022arXiv220806896H}. That said, the argument was certainly inspired by the recent developments in the restricted projection problem.

\subsection{Proof outline for Theorem \ref{main}}\label{s:outline} The proof of Theorem \ref{main} is mainly based on two ingredients. The first one is a \emph{point-line duality} principle between horizontal lines in $\He$, and $\R^{3}$. To describe this principle, let $\mathcal{L}_{\He}$ be the family of all horizontal lines in $\He$, and let $\mathcal{L}_{\mathcal{C}}$ be the family of all lines in $\R^{3}$ which are parallel to some line contained in a conical surface $\mathcal{C}$. In Section \ref{s:duality}, we show that there exist maps $\ell \colon \R^{3} \to \mathcal{L}_{\He}$ and $\ell^{\ast} \colon \He \to \mathcal{L}_{\mathcal{C}}$ (whose ranges cover almost all of $\mathcal{L}_{\He}$ and $\mathcal{L}_{\mathcal{C}}$) which preserve \emph{incidence relations} in the following way:
\begin{displaymath} q \in \ell(p) \quad \Longleftrightarrow \quad p \in \ell^{\ast}(q), \qquad p \in \R^{3}, \, q \in \He. \end{displaymath}
Thus, informally speaking, incidence-geometric questions between
points in $\He$ and lines in $\mathcal{L}_{\He}$ can always be
transformed into incidence-geometric questions between points in
$\R^{3}$ and lines in $\mathcal{L}_{\mathcal{C}}$. The point-line duality principle described here was used implicitly by Liu \cite{MR4439466} to study Kakeya sets (formed by
horizontal lines) in $\He$. However, making the principle explicit has already proved very useful since the first version of this paper appeared: we used it in \cite{2022arXiv221009955F} to study Kakeya sets associated with $SL(2)$-lines in $\R^{3}$, and Harris \cite{2023arXiv230104645H} used it to treat the case $\Hd K > 3$ of Theorem \ref{mainSimple} (in this case the projections $\pi_{e}(K)$ turn out to have positive measure almost surely).

The question about vertical projections in $\He$ can -- after suitable discretisation -- be interpreted as an incidence geometric problem between points in $\He$ and lines in $\mathcal{L}_{\He}$. It can therefore be transformed into an incidence-geometric problem between points in $\R^{3}$ and lines in $\mathcal{L}_{\mathcal{C}}$. Which problem is this? It turns out that while the dual $\ell^{\ast}(p)$ of a point $p \in \He$ is a line in $\mathcal{L}_{\mathcal{C}}$, the dual $\ell^{\ast}(B_{\He})$ of a Heisenberg $\delta$-ball resembles an \emph{$\delta$-plate} in $\R^{3}$ -- a rectangle of dimensions $1 \times \delta \times \delta^{2}$ tangent to $\mathcal{C}$. So, the task of proving Theorem \ref{mainMeasure} (hence Theorem \ref{main}) is (roughly) equivalent to the task of solving an incidence-geometric problem between points in $\R^{3}$, and family of $\delta$-plates.

Moreover: the plates in our problem appear as duals of certain Heisenberg $\delta$-balls, approximating a $t$-dimensional set $K \subset \He$, with $0 \leq t \leq 3$. Consequently, the plates can be assumed to satisfy a $t$-dimensional "non-concentration condition" relative to the metric $d_{\He}$. In common jargon, the plate family is a \emph{$(\delta,t)$-set} relative to $d_{\He}$.

In \cite{MR4151084}, Guth, Wang, and Zhang proved the \emph{sharp (reverse) square function estimate for the cone in $\R^{3}$}. A key component in their proof was a new incidence-geometric ("Kakeya") estimate \cite[Lemma 1.4]{MR4151084} for points and $\delta$-plates in $\R^{3}$ (see Section \ref{s:GWZ} for the details). While this was not relevant in \cite{MR4151084}, it turns out that the incidence estimate in \cite[Lemma 1.4]{MR4151084} interacts perfectly with a $(\delta,3)$-set condition relative to $d_{\He}$. This allows us to prove, roughly speaking, that the vertical projections of $3$-Frostman measures on $\He$ have $L^{2}$-densities. See Corollary \ref{cor1} for a more precise statement.

For $0 \leq t < 3$, the $(\delta,t)$-set condition relative to $d_{\He}$ no longer interacts so well with \cite[Lemma 1.4]{MR4151084}. However, we were able to (roughly speaking) reduce Theorem \ref{mainMeasure} for $(\delta,t)$-sets, $0 \leq t \leq 3$, to the special case $t = 3$. This argument is explained in Section \ref{s:discretisation}, so we omit the discussion here.

\subsection*{Acknowledgements} We thank the reviewer for a careful reading of the manuscript, and for providing us with helpful comments. 

\section{Preliminaries on the Heisenberg group}\label{s:heis}

We briefly introduce the Heisenberg group and relevant related concepts. A more thorough
introduction to the geometry of the Heisenberg group can be found
in many places, for instance in the monograph \cite{CDPT}.

The \emph{Heisenberg group} $\mathbb{H}=(\mathbb{R}^3,\cdot)$ is
the set $\mathbb{R}^3$ equipped with the non-commutative group
product defined by
\begin{displaymath}
(x,y,t)\cdot (x',y',t')=
\left(x+x',y+y',t+t'+\tfrac{1}{2}(xy'-yx')\right).
\end{displaymath}
The \emph{Heisenberg dilations} are the group automorphisms
$\delta_{\lambda}$, $\lambda>0$, defined by
\begin{displaymath}
\delta_{\lambda}(x,y,t)=(\lambda x,\lambda y,\lambda^2 t).
\end{displaymath}
The group product gives rise to projection-type mappings onto
subgroups that are invariant under Heisenberg dilations. For $e\in
S^1$, we define the \emph{horizontal subgroup}
\begin{displaymath}
\mathbb{L}_e := \{(se,0):\, s\in \mathbb{R}\}.
\end{displaymath}
The \emph{vertical subgroup} $\mathbb{W}_e$  is the Euclidean
 orthogonal complement of $\mathbb{L}_e$ in $\mathbb{R}^3$; in
 particular it is a plane containing the vertical axis.
Every point $p\in \mathbb{H}$ can be written in a unique way as a
product $p =
 p_{\mathbb{W}_e}\cdot p_{\mathbb{L}_e}$ with $p_{\mathbb{W}_e}\in
 \mathbb{W}_e$ and $p_{\mathbb{L}_e}\in \mathbb{L}_e$.
The \emph{vertical Heisenberg projection} onto the vertical plane
$\mathbb{W}_e$ is the map
\begin{displaymath}
\pi_e:\mathbb{H} \to \mathbb{W}_e,\quad p= p_{\mathbb{W}_e}\cdot
p_{\mathbb{L}_e} \mapsto p_{\mathbb{W}_e}.
\end{displaymath}
The vertical projection to the $xt$-plane $\{(x,0,t) : x,t \in \R\}$ will play a special role; this projection will be denoted $\pi_{xt}$, and it has the explicit formula stated in \eqref{projections}.  Preliminaries about Heisenberg projections can be found for
instance in \cite{MR2789472,MR2955184,MR3047423}. These mappings
have turned out to play an important role in geometric measure
theory of the Heisenberg group endowed with a left-invariant
non-Euclidean metric. The \emph{Kor\'{a}nyi metric}
$d_{\mathbb{H}}$ is defined by
\begin{displaymath}
d_{\mathbb{H}}(p,q):=\|q^{-1}\cdot p\|,
\end{displaymath}
where $\|\cdot\|$ is the \emph{Kor\'{a}nyi norm} given by
\begin{displaymath}
\|(x,y,t)\| = \sqrt[4]{(x^2+y^2)^2 + 16 t^2}.
\end{displaymath}
We will use the symbol $B_{\mathbb{H}}(p,r)$ to denote the ball
centered at $p$ with radius $r$ with respect to the Kor\'{a}nyi
metric. Balls centred at the origin are denoted $B_{\He}(r)$. All vertical planes $\mathbb{W}_e$, $e\in S^1$, equipped
with $d_{\mathbb{H}}$ are isometric to each other via rotations of
$\mathbb{R}^3$ about the vertical axis. The Heisenberg dilations
are similarities with respect to $d_{\mathbb{H}}$, and it is easy
to see that $(\mathbb{H},d_{\mathbb{H}})$ is a $4$-regular space,
while the vertical subgroups $\mathbb{W}_e$ are $3$-regular with
respect to $d_{\mathbb{H}}$. Moreover, there exists a constant
$0<c<\infty$, independent of $e$, such that under the obvious
identification of $\mathbb{W}_e$ with $\mathbb{R}^2$, the
restriction of the $3$-dimensional Hausdorff measure
$\mathcal{H}^3$ to $\mathbb{W}_e$ agrees with the $2$-dimensional
Lebesgue measure $\mathrm{Leb}$ on $\mathbb{R}^2$ up to the
multiplicative constant $c$.

Vertical projections are neither group homomorphisms nor Lipschitz
mappings with respect $d_{\mathbb{H}}$. However, they behave well
with respect to the Lebesgue measure on vertical planes. Namely,
for every Borel set $E \subset \He$, we have that
\begin{equation}
\label{projLeb} \mathrm{Leb}\left(\pi_e(p\cdot E)\right)=
\mathrm{Leb}\left(\pi_e(E)\right), \qquad p\in
\mathbb{R}^3,\,e\in S^1,
\end{equation}
see the formula at the bottom of page 1970 in the proof of
\cite[Lemma 2.20]{MR3511465}.

\section{Proof of Theorem \ref{mainLowDimIntro}}\label{s:Dim2}

In this section, we prove Theorem \ref{mainLowDimIntro}, and therefore the cases $\Hd K \in [0,2]$ of Theorem \ref{mainSimple}. Further, Theorem \ref{mainLowDimIntro} will be inferred from a more general statement, Theorem \ref{mainTransversal}, modelled after \cite[Theorem 1.3]{2022arXiv220702259P}. We first
discuss Theorem \ref{mainTransversal}, and then explain in Section
\ref{ss:RestProjToHei} how it can be applied to deduce Theorem \ref{mainLowDimIntro}.

\subsection{Projections induced by cinematic functions}
We start by introducing terminology from \cite[Definition
1.6]{2022arXiv220702259P} which will be needed for the formulation
of Theorem \ref{mainTransversal}.

\begin{definition}[Cinematic family]\label{d:cinematic}
 Let $I \subset \R$ be a compact interval, and let $\mathcal{F} \subset C^{2}(I)$ be a family of functions satisfying the following conditions:
\begin{enumerate}
\item $I$ is a compact interval, and $\mathcal{F}$ has finite
diameter in $(C^{2}(I),\|\cdot\|_{C^{2}(I)})$. \item
$(\mathcal{F},\|\cdot\|_{C^{2}(I)})$ is a doubling metric space.
\item For all $f,g \in \mathcal{F}$, we have
\begin{displaymath} \inf_{\theta \in I} |f(\theta) - g(\theta)| + |f'(\theta) - g'(\theta)| + |f''(\theta) - g''(\theta)| \gtrsim \|f - g\|_{C^{2}(I)}. \end{displaymath}
\end{enumerate}
Then, $\mathcal{F}$ is called a \emph{cinematic family}.
\end{definition}

The following projection theorem is modelled after \cite[Theorem
1.3]{2022arXiv220702259P}:

\begin{thm}\label{mainTransversal} Let $L > 0$, let $I \subset \R$ be a compact interval, and let $\{\rho_{\theta}\}_{\theta \in I}$ be a family of $L$-Lipschitz maps $\rho_{\theta} \colon B \to \R$, where $B \subset \R^{3}$ is a ball. For $p \in B$, define the function $f_{p} \colon I \to \R$ by $f_{p}(\theta) := \rho_{\theta}(p)$. Assume that $p \mapsto f_{p}$ is a bilipschitz embedding $B \to C^{2}(I)$, and assume that $\mathcal{F} = \{f_{p} : p \in B\}$ is a cinematic family.

Then, the projections $\{\rho_{\theta}\}_{\theta \in I}$ satisfy
\eqref{form46}: if $K \subset \R^{3}$ is a Borel set, then
\begin{displaymath} \dim_{\mathrm{E}} \{\theta \in I : \dim_{\mathrm{E}} \rho_{\theta}(K) \leq s\} \leq s, \qquad 0 \leq s < \min\{\dim_{\mathrm{E}} K,1\}. \end{displaymath}
\end{thm}

We only sketch the proof of Theorem \ref{mainTransversal} since it
is virtually the same as the proof of \cite[Theorem
1.3]{2022arXiv220702259P}: this is the special case of Theorem
\ref{mainTransversal}, where
\begin{equation}\label{form60} f_{p}(\theta) = \rho_{\theta}(p) := \gamma(\theta) \cdot p, \qquad p \in \R^{3}, \end{equation}
and $\gamma \colon I \to S^{2}$ parametrises a curve on $S^{2}$
satisfying $\spa\{\gamma,\dot{\gamma},\ddot{\gamma}\} = \R^{3}$
(this condition is needed to guarantee that the family $\{f_{p} :
p \in B\}$ is cinematic for every ball $B \subset \R^{3}$, see the
proof of \cite[Proposition 2.1]{2022arXiv220702259P}).

The proof of \cite[Theorem 1.3]{2022arXiv220702259P} is based on a
reduction to \cite[Theorem 1.7]{2022arXiv220702259P}. This is a
"Kakeya-type" estimate concerning $\delta$-neighbourhoods of
graphs of cinematic functions. More precisely, \cite[Theorem
1.7]{2022arXiv220702259P} is only used via \cite[Proposition
2.1]{2022arXiv220702259P}, a special case of \cite[Theorem
1.7]{2022arXiv220702259P} concerning the cinematic family
$\{\theta \mapsto \gamma(\theta) \cdot p\}_{p \in B}$. We
formulate a more general version of this proposition below: the
only difference is that the cinematic family $\{\theta \mapsto
\gamma(\theta) \cdot p\}_{p \in B}$ is replaced by the family
$\{\theta \mapsto \rho_{\theta}(p)\}_{p \in B}$ relevant for
Theorem \ref{mainTransversal}:
\begin{proposition}\label{prop4} Fix $\epsilon > 0$ and $0 < \alpha \leq \zeta \leq 1$. Let $I \subset \R$ be a compact interval, let $B \subset \R^{3}$ be a ball, and let $\rho_{\theta} \colon B \to \R$ be a family of uniformly Lipschitz functions with the properties assumed in Theorem \ref{mainTransversal}: thus, $\mathcal{F} = \{f_{p} : p \in B\}$ is a cinematic family, and the map $p \mapsto f_{p}$ is a bilipschitz embedding $B \to C^{2}(I)$, where $f_{p}(\theta) := \rho_{\theta}(p)$. Then there exists $\delta_{0} > 0$ such that the following holds for all $\delta \in (0,\delta_{0}]$:

Let $E \subset \R^{2}$ be a
$(\delta,\alpha;\delta^{-\epsilon})_{1} \times
(\delta,\alpha;\delta^{-\epsilon})_{1}$ quasi-product. Let
$Z_{\delta} \subset B$ be a $\delta$-separated set that satisfies
\begin{equation}\label{form62} |Z_{\delta} \cap B(p,r)| \leq \delta^{-\epsilon}(r/\delta)^{\zeta}, \qquad p \in \R^{3}, \, r \geq \delta. \end{equation}
Then
\begin{displaymath} \int_{E} \Big( \sum_{p \in Z_{\delta}} \mathbf{1}_{\Gamma_{p}^{\delta}} \Big)^{3/2} \leq \delta^{2 - \alpha/2 - \zeta/2 - C\epsilon}|Z_{\delta}|, \end{displaymath}
where $C > 0$ is absolute, and $\Gamma_{p}^{\delta}$ is the
$\delta$-neighbourhood of the graph of $f_{p}$. \end{proposition}
\begin{proof} The proof of \cite[Proposition 2.1]{2022arXiv220702259P} is easy (given \cite[Theorem 1.7]{2022arXiv220702259P}), but the proof of Proposition \ref{prop4} is almost trivial. Indeed, the first part in the proof of \cite[Proposition 2.1]{2022arXiv220702259P} is to verify that the family $\{\theta \mapsto \rho_{\theta}(p)\}_{p \in B}$ is cinematic in the case $\rho_{\theta}(p) = \gamma(\theta) \cdot p$, but this is already a part of our hypothesis. The second part in the proof of \cite[Proposition 2.1]{2022arXiv220702259P} is to verify that $p \mapsto f_{p}$ is a bilipschitz embedding $B \to C^{2}(I)$, and this is -- again -- part of our hypothesis. In other words, all the work in the proof of \cite[Proposition 2.1]{2022arXiv220702259P} has been made part of the hypotheses of Proposition \ref{prop4}.  \end{proof}

The reduction from \cite[Theorem 1.3]{2022arXiv220702259P} to
\cite[Proposition 2.1]{2022arXiv220702259P} (in our case from
Theorem \ref{mainTransversal} to Proposition \ref{prop4}) is
presented in \cite[Sections 2.1-2.4]{2022arXiv220702259P}, and
does not use the special form \eqref{form60} (for example the
linearity) of the maps $\rho_{\theta} \colon \R^{3} \to \R$: it is
only needed that
\begin{enumerate}
\item the maps $\rho_{\theta}$ are uniformly Lipschitz, for
$\theta \in I$, \item $\sup_{p \in B} \sup_{\theta \in I}
|\partial_{\theta} \rho_{\theta}(p)| < \infty$.
\end{enumerate}
Property (1) is assumed in Theorem \ref{mainTransversal}, whereas
property (2) follows from the assumption that the family
$\mathcal{F}$ is cinematic (and in particular a bounded subset of
$C^{2}(I)$).

The argument in \cite[Sections 2.1-2.4]{2022arXiv220702259P} is
extremely well-written, and our notation is deliberately the same,
so we will not copy the whole proof. We only make a few remarks,
below. If the reader is unfamiliar with the ideas involved, we
warmly recommend reading first the heuristic section \cite[Section
1.2]{2022arXiv220702259P}.

\begin{proof}[Proof sketch of Theorem \ref{mainTransversal}] The argument in \cite[Section 2.1]{2022arXiv220702259P} can be copied verbatim; nothing changes. The most substantial change occurs in \cite[Section 2.2]{2022arXiv220702259P}. Namely, \cite[(2.10)]{2022arXiv220702259P} uses the fact (true in \cite{2022arXiv220702259P}) that the $\rho_{\theta}$-image of a $\delta$-cube $Q \subset \R^{3}$ has length $|\rho_{\theta}(Q)| \gtrsim \delta$. For the general Lipschitz maps $\rho_{\theta}$ in Theorem \ref{mainTransversal} this may not be the case; it would be true for the special maps $\rho_{\theta}$ needed in Theorem \ref{mainLowDim}, so also this part of \cite{2022arXiv220702259P} would work verbatim for these maps.  However, even in the generality of Theorem \ref{mainTransversal} the problem can be completely removed: one only needs to replace every occurrence of $\rho_{\theta}(Q)$ in \cite[Section 2.2]{2022arXiv220702259P} by an interval
\begin{displaymath} I_{\theta}(Q) := [\rho_{\theta}(z_{Q}) - \delta,\rho_{\theta}(z_{Q}) + \delta] \end{displaymath}
of length $\sim \delta$ centred at $\rho_{\theta}(z_{Q})$, where
$z_{Q} \in Q$ is the centre of $Q$. Since $\rho_{\theta}(Q)$ only
appears as a "tool" in \cite[Section 2.2]{2022arXiv220702259P},
the rest of the argument will remain unchanged. Let us, however,
discuss what changes in \cite[Section 2.2]{2022arXiv220702259P}
when $\rho_{\theta}(Q)$ is replaced by $I_{\theta}(Q)$. We assume
familiarity with the notation in \cite{2022arXiv220702259P}.

First and foremost, \cite[(2.9)]{2022arXiv220702259P} remains
valid: whenever $Q \in \mathcal{Q}$ is a cube that intersects
$\rho_{\theta}^{-1}(G_{\theta})$, then
$\dist(\rho_{\theta}(z_{Q}),G_{\theta}) \lesssim L\delta$ by our
assumption that the maps $\rho_{\theta}$ are $L$-Lipschitz.
Therefore,
\begin{displaymath} I_{\theta}(Q) \subset G_{\theta}' := N_{L\delta}(G_{\theta}). \end{displaymath}
This gives \cite[(2.9)]{2022arXiv220702259P} with the slightly
modified definition of $G_{\theta}'$, stated above. Consequently,
also the version of \cite[(2.10)]{2022arXiv220702259P} is true
where $\rho_{\theta}(Q)$ is replaced by $I_{\theta}(Q)$: here the
length bound $|I_{\delta}(Q)| \gtrsim \delta$ is used. Finally, to
deduce \cite[(2.13)]{2022arXiv220702259P} from
\cite[(2.10)]{2022arXiv220702259P}, we need to know that
\cite[(2.12)]{2022arXiv220702259P} remains valid when
$\rho_{\theta}(Q)$ is replaced with $I_{\theta}(Q)$. This is
clear: if $y \in I_{\theta}(Q)$, then $|y - \rho_{\theta}(z_{Q})|
\leq \delta$ by definition, and therefore $(\theta,y) \in
\Gamma_{z_{Q}}^{\delta}$, where
\begin{displaymath} \Gamma_{z} = \{(\theta,\rho_{\theta}(z)) : \theta \in I\}, \qquad z \in B, \end{displaymath}
is the analogue of \cite[(1.12)]{2022arXiv220702259P}, and
$\Gamma_{z}^{\delta}$ is the $\delta$-neighbourhood of
$\Gamma_{z}$. We have now verified
\cite[(2.13)]{2022arXiv220702259P}. The intervals
$\rho_{\theta}(Q)$ or $I_{\theta}(Q)$ play no further role in the
proof. The rest of \cite[Section 2.2]{2022arXiv220702259P} works
verbatim.

The same is also true for \cite[Section 2.3]{2022arXiv220702259P}:
the argument is fairly abstract down to
\cite[(2.19)]{2022arXiv220702259P}, where it is needed that
$\sup_{p \in B} \sup_{\theta \in I} |\partial_{\theta}
\rho_{\theta}(p)| < \infty$. The maps $\rho_{\theta}$ in Theorem
\ref{mainTransversal} satisfy this property automatically, as
noted in (2) above.

Finally, we arrive at the short \cite[Section
2.4]{2022arXiv220702259P}. The only difference is that we need to
apply Proposition \ref{prop4} in place of \cite[Proposition
2.1]{2022arXiv220702259P}. This completes the proof of Theorem
\ref{mainTransversal}. \end{proof}

\subsection{From vertical projections to cinematic functions}\label{ss:RestProjToHei}
We explain how the general projection result, Theorem
\ref{mainTransversal}, can be applied to prove Theorem
\ref{mainLowDim}, which concerns the special projections $\rho_{e} = \pi_{T} \circ \pi_{e}$. Recall that $\pi_{e}$ is the vertical projection
to the plane $\W_{e} = e^{\perp}$. For $e = (e_{1},e_{2}) \in
S^{1}$,
 we write $J(e) := (-e_{2},e_{1}) \in S^{1} \cap e^{\perp}$ is the counterclockwise rotation of $e$ by $\pi/2$. With this notation, the map $\pi_{e}$ has the explicit formula
\begin{equation}\label{form49} \pi_{e}(z,t) = (\langle z,Je \rangle, t + \tfrac{1}{2} \langle z, e \rangle \langle z,Je \rangle), \end{equation}
where $\langle \cdot,\cdot \rangle$ is the Euclidean dot product in $\R^{2}$. In the formula \eqref{form49}, we have also identified each plane $\W_{e}$ with $\R^{2}$ via the map $(yJe,t) \cong (y,t)$. It is worth noting that the distance $d_{\He}$ restricted to the plane $\W_{e}$ (for $e \in S^{1}$ fixed) is bilipschitz equivalent to the parabolic distance on $\R^{2}$, namely $d_{\mathrm{par}}((x,s),(y,t)) = |x - y| + \sqrt{|s - t|}$.

With the explicit expression \eqref{form49} in hand, the nonlinear projections $\rho_{e} = \pi_{T} \circ \pi_{e}$ introduced in \eqref{form63} have the following formula: 
\begin{displaymath} \rho_{e}(z,t) = t + \tfrac{1}{2} \langle z, e \rangle \langle z,Je \rangle, \qquad (z,t) \in \R^{2} \times \R, \, e \in S^{1}, \end{displaymath}
By a slight abuse of notation, we write "$d_{\He}$" for the square root metric on $\R$: thus $d_{\He}(s,t) := \sqrt{s - t}$. The projection $\pi_{T}$ restricted to any fixed plane $\mathbb{W}_{e}$ is a Lipschitz map $(\W_{e},d_{\He}) \to (\R,d_{\He})$, even though $\pi_{T}$ is not "globally" a Lipschitz map $(\He,d_{\He}) \to (\R,d_{\He})$. Therefore $\Hd \pi_{e}(K) \geq \Hd \rho_{e}(K)$ for all $e \in S^{1}$, and the cases $\Hd K \in [0,2]$ of Theorem \ref{mainSimple} follow from Theorem \ref{mainLowDimIntro}, whose contents are repeated here:
\begin{thm}\label{mainLowDim} Let $K \subset \R^{3}$ be Borel, and let $0 \leq s < \min\{\dim_{\mathrm{E}} K,1\}$. Then,
\begin{equation}\label{form46} \dim_{\mathrm{E}} \{e \in S^{1} : \dim_{\mathrm{E}} \rho_{e}(K) \leq s\} \leq s. \end{equation}
As a consequence, for every $0 \leq s < \min\{\Hd K,2\}$,
\begin{equation}\label{form47} \dim_{\mathrm{E}} \{e \in S^{1} : \Hd \rho_{e}(K) \leq s\} \leq \tfrac{s}{2}. \end{equation}
In particular, $\Hd \rho_{e}(K) \geq \min\{\Hd K,2\}$ for $\mathcal{H}^{1}$ almost every $e \in S^{1}$.
\end{thm}

\begin{remark}\label{rem3} We explain why \eqref{form46} implies \eqref{form47}. It is well-known that
\begin{displaymath} \Hd K \leq 2 \dim_{\mathrm{E}} K.  \end{displaymath}
for all sets $K \subset \He$. This simply follows from the fact that the identity map $(\He,d_{\mathrm{Euc}}) \to (\He,d_{\He})$ is locally $\tfrac{1}{2}$-H\"older continuous. Therefore, if $0 \leq s < \min\{\Hd K,2\}$, as in \eqref{form47}, we have $0 \leq \tfrac{s}{2} < \min\{\dim_{\mathrm{E}} K,1\}$, and \eqref{form46} is applicable. Since
\begin{displaymath} \{e \in S^{1} : \Hd \rho_{e}(K) \leq s\} = \{e \in S^{1} : \dim_{\mathrm{E}} \rho_{e}(K) \leq \tfrac{s}{2}\} \end{displaymath}
(the square root metric on $\R$ doubles Euclidean dimension), we have
\begin{displaymath} \dim_{\mathrm{E}} \{e \in S^{1} : \Hd \rho_{e}(K) \leq s\} = \dim_{\mathrm{E}} \{e \in S^{1} : \dim_{\mathrm{E}} \rho_{e}(K) \leq \tfrac{s}{2}\} \stackrel{\eqref{form46}}{\leq} \tfrac{s}{2}. \end{displaymath}
This is what we claimed in \eqref{form47}. \end{remark}

For the remainder of this section, we focus on proving the Euclidean statement \eqref{form46}.
This is chiefly based on verifying that the projections $\rho_{e} \colon \R^{3} \to \R$
 give rise to a \emph{cinematic family of functions}, as in
 Definition \ref{d:cinematic}.
Let us introduce the relevant cinematic family. We re-parametrise
the projections $\rho_{e}$, $e \in S^{1}$, as $\rho_{\theta}$,
$\theta \in \R$, where
\begin{displaymath} \rho_{\theta} := \rho_{e(\theta)}, \qquad e(\theta) := (\cos \theta,\sin \theta). \end{displaymath}
With this notation, we define the following functions $f_{p} \colon \R \to \R$, $p \in \R^{3}$:
\begin{equation}\label{form53} f_{p}(\theta) := \rho_{\theta}(p) := t + \tfrac{1}{2}\langle z,e(\theta) \rangle \langle z,Je(\theta) \rangle, \qquad p = (z,t) \in \R^{3}. \end{equation}
\begin{proposition}\label{prop:cinematic} Let $p_{0} \in \R^{3} \, \setminus \, \{(0,0,t) : t \in \R\}$.
Then, there exists a radius $r = r(p_{0}) > 0$ such that $\mathcal{F}(B(p_{0},r)) := \{f_{p} : p \in B(p_{0},r)\}$ is a cinematic family.
 \end{proposition}
The compact interval appearing in conditions (1)-(3) of Definition
\ref{d:cinematic} can be taken to be $[0,2\pi]$ -- this makes no
difference, since the functions $f_{p}$ are $2\pi$-periodic. It
turns out that the conditions (1)-(2) are satisfied for the family
$\mathcal{F}(B)$, whenever $B \subset \R^{3}$ is an arbitrary
ball. To verify condition (3), we will need to assume that $B$
lies outside the $t$-axis; we will return to this a little later.
We first compute the derivatives of the functions in
$\mathcal{F}$. For $f_{p} \in \mathcal{F}$, we have
\begin{displaymath} f_{p}'(\theta) = \tfrac{1}{2}\langle z,e'(\theta) \rangle \langle z,Je(\theta) \rangle + \tfrac{1}{2} \langle z,e(\theta) \rangle\langle z,Je'(\theta) \rangle.  \end{displaymath}
This expression can be further simplified by noting that $e'(\theta) = Je(\theta)$, and $Je'(\theta) = -e(\theta)$. Therefore,
\begin{equation}\label{form51} f_{p}'(\theta) = \tfrac{1}{2}\langle z,Je(\theta) \rangle^{2} - \tfrac{1}{2}\langle z,e(\theta) \rangle^{2}. \end{equation}
From this expression, we may compute the second derivative:
\begin{equation}\label{form52} f_{p}''(\theta) = \langle z,Je(\theta) \rangle \langle z,Je'(\theta) \rangle - \langle z,e(\theta) \rangle \langle z,e'(\theta) \rangle = -2\langle z,e(\theta) \rangle \langle z,Je(\theta) \rangle. \end{equation}
The formulae \eqref{form53}-\eqref{form52} immediately show that the map $p \mapsto f_{p}$ is locally Lipschitz:
\begin{equation}\label{form58} \sup_{\theta \in \R} |f_{p}(\theta) - f_{q}(\theta)| + |f_{p}'(\theta) - f_{q}'(\theta)| + |f_{p}''(\theta) - f_{q}''(\theta)| \lesssim_{B} |p - q|, \qquad p,q \in B. \end{equation}
This implies conditions (1)-(2) in  Definition \ref{d:cinematic}
for the family $\mathcal{F}(B)$. Regarding condition (3) in
Definition \ref{d:cinematic}, we claim the following:

\begin{proposition}\label{prop3} If $p_{0} \in \R^{3} \, \setminus \, \{(0,0,t) : t \in \R\}$, there exists a radius $r = r(p_{0}) > 0$ and a constant $c = c(p_{0}) > 0$ such that
\begin{equation}\label{form54} |f_{p}(\theta) - f_{q}(\theta)| + |f_{p}'(\theta) - f_{q}'(\theta)| + |f_{p}''(\theta) - f_{q}''(\theta)| \geq c|p - q| \end{equation}
for all $p,q \in B(p_{0},r)$ and $\theta \in \R$. \end{proposition}

We start with the following lemma:
\begin{lemma}\label{lemma3} For every $p_{0} \in \R^{3} \, \setminus \, \{(0,0,t) : t \in \R\}$ there exists a constant $c > 0$ and a radius $r > 0$ such that the following holds:
\begin{equation}\label{form55} |f_{p}(0) - f_{q}(0)| + |f_{p}'(0) - f_{q}'(0)| + |f_{p}''(0) - f_{q}''(0)| \geq c|p - q|, \qquad p,q \in B(p_{0},r). \end{equation}
\end{lemma}

\begin{proof} Recall that $e(0) = (1,0)$ and $Je(0) = (0,1)$. We then define $F \colon \R^{3} \to \R^{3}$ by
\begin{displaymath} F(p) := (f_{p}(0),f_{p}'(0),f_{p}''(0)) = (t + \tfrac{1}{2}z_{1}z_{2},\tfrac{1}{2}(z_{2}^{2} - z_{1}^{2}),-2z_{1}z_{2}), \qquad p = (z,t) \in \R^{3}. \end{displaymath}
Then, we note that $|\mathrm{det} DF(p)| = 2|z|^{2}$, so in particular the Jacobian of $F$ is non-vanishing outside the $t$-axis. Now \eqref{form55} follows from the inverse function theorem. \end{proof}

We then prove Proposition \ref{prop3}:

\begin{proof}[Proof of Proposition \ref{prop3}] To deduce \eqref{form54} from \eqref{form55}, we record the following rotation invariance:
\begin{equation}\label{form56} f^{(k)}_{R_{\varphi}(p)}(\theta + \varphi) = f^{(k)}_{p}(\theta), \qquad p \in \R^{3}, \, \theta,\varphi \in \R. \end{equation}
Here $R_{\varphi}(z,t) := (e^{i\varphi}z,t)$ is a counterclockwise rotation around the $t$-axis. The proof is evident from the formulae \eqref{form53}-\eqref{form52}, and noting that
\begin{displaymath} \langle e^{i\varphi}z,e(\theta + \varphi)\rangle = \langle z,e(\theta) \rangle \quad \text{and} \quad \langle e^{i\varphi}z,Je(\theta + \varphi) \rangle = \langle z,Je(\theta) \rangle. \end{displaymath}
Now we are in a position to conclude the proof of \eqref{form54}. Fix $p_{0} \in \R^{3} \, \, \setminus \, \{(0,0,t) : t \in \R\}$ and $\theta_{0} \in \R$. Then, apply Lemma \ref{lemma3} to the point
\begin{displaymath} R_{-\theta_{0}}(p_{0}) \in \R^{3} \, \setminus \, \{(0,0,t) : t \in \R\}. \end{displaymath}
This yields a constant $c = c(p_{0},\theta_{0}) > 0$ and a radius $r_{0} = r_{0}(p_{0},\theta_{0}) > 0$ such that
\begin{equation}\label{form57} |f_{p}(0) - f_{q}(0)| + |f_{p}'(0) - f_{q}'(0)| + |f_{p}''(0) - f_{q}''(0)| \geq c|p - q| \end{equation}
for all $p,q \in B(R_{-\theta_{0}}(p_{0}),2r_{0})$. Next, we choose $I(\theta_{0}) = [\theta_{0} - r_{1},\theta_{0} + r_{1}]$ to be a sufficiently short interval around $\theta_{0}$ such that the following holds:\begin{displaymath} R_{-\theta}(p),R_{-\theta}(q) \in B(R_{-\theta_{0}}(p_{0}),2r_{0}), \qquad p,q \in B(p_{0},r_{0}), \, \theta \in I(\theta_{0}). \end{displaymath}
Then, it follows from a combination of \eqref{form56} and \eqref{form57} that
\begin{displaymath} \sum_{k = 0}^{2} |f_{p}^{(k)}(\theta) - f_{q}^{(k)}(\theta)| \stackrel{\eqref{form56}}{=} \sum_{k = 0}^{2} |f_{R_{-\theta}(p)}^{(k)}(0) - f_{R_{-\theta}(q)}^{(k)}(0)| \stackrel{\eqref{form57}}{\geq} c|R_{-\theta}(p) - R_{-\theta}(q)| = c|p - q| \end{displaymath}
for all $p,q \in B(p_{0},r_{0})$ and all $\theta \in I(\theta_{0})$. This completes the proof of \eqref{form54} of all $\theta \in I(\theta_{0})$. To extend the argument of all $\theta \in \R$, note that the functions $f_{p}$, and all of their derivatives, are $2\pi$-periodic. So, it suffices to show that \eqref{form54} holds for $\theta \in [0,2\pi]$. This follows by compactness from what we have already proven, by covering $[0,2\pi]$ by finitely many intervals of the form $I(\theta_{0})$, and finally defining "$r$" and "$c$" to be the minima of the constants $r(p_{0},\theta_{0})$ and $c(p_{0},\theta_{0})$ obtained in the process. \end{proof}

Proposition \ref{prop:cinematic} now follows from Proposition
\ref{prop3}, and the discussion above it (where we verified Definition \ref{d:cinematic}(1)-(2)). We then
conclude the proof of Theorem~\ref{mainLowDim}:

\begin{proof}[Proof of Theorem \ref{mainLowDim}] Given Remark \ref{rem3}, it suffices to prove \eqref{form46},
which will be a consequence of Theorem \ref{mainTransversal}.
Indeed, since the projections $\rho_{e}$ are isometries on the
$t$-axis, we may assume that
\begin{displaymath} \dim_{E} (K \, \setminus \, \{(0,0,t) : t \in \R\}) = \dim_{E} K. \end{displaymath}
Consequently, for  $\epsilon > 0$, we may fix a point $p_{0} \in K$ outside the $t$-axis such that
\begin{equation}\label{form59} \dim_{\mathrm{E}} (K \cap B(p_{0},r)) > \dim_{\mathrm{E}} K - \epsilon, \qquad r > 0. \end{equation}
Apply Proposition \ref{prop:cinematic} to find a radius $r > 0$
such that the family of functions $\mathcal{F} :=
\mathcal{F}(B(p_{0},r))$ is cinematic. It follows from a
combination of \eqref{form58}  and Proposition \ref{prop3} that $p
\mapsto f_{p}$ is a bilipschitz embedding $B \to C^{2}(\R)$.
 Therefore Theorem \ref{mainTransversal} is applicable: for every $0 \leq s < \min\{\dim_{\mathrm{E}} (K \cap B(p_{0},r)),1\}$ we have
\begin{displaymath} \dim_{\mathrm{E}} \{\theta \in [0,2\pi] : \dim_{\mathrm{E}} \rho_{\theta}(K \cap B(p_{0},r)) \leq s\} \leq s. \end{displaymath}
Now \eqref{form46} follows from \eqref{form59} by letting $\epsilon \to 0$. \end{proof}

\section{Duality between horizontal lines and $\mathbb{R}^{3}$}\label{s:duality}

This section contains preliminaries to prove Theorem \ref{main}. Most importantly, we introduce a notion of \emph{duality} that associates to points
and horizontal lines in $\mathbb{H}$ certain lines and points in
$\mathbb{R}^3$. The lines in $\R^{3}$ will be \emph{light rays} -- translates of lines on a fixed conical surface. To define these, we let $\mathcal{C}_{0}$ be the vertical cone
\begin{displaymath}
\mathcal{C}_0=\{(z_1,z_2,z_3)\in \mathbb{R}^3:\, z_1^2 + z_2^2 = z_3^2\} ,
\end{displaymath}
and we denote by $\mathcal{C}$ the ($45^{\circ}$) rotated cone
\begin{displaymath}
\mathcal{C}=R(\mathcal{C}_0)=\{(z_1,z_2,z_3)\in \mathbb{R}^3:\, z_2^2 = 2 z_1 z_3\},
\end{displaymath}
where
$R(z_1,z_2,z_3)=\left((z_1+z_3)/\sqrt{2},z_2,(-z_1+z_3)/\sqrt{2}\right)$.
The cone $\mathcal{C}$ is foliated by lines
\begin{equation}\label{eq:LineCone}
L_y = \mathrm{span}_{\mathbb{R}} (1,-y,y^2/2),\quad y\in
\mathbb{R},
\end{equation}
cf. the proof of \cite[Theorem 1.2]{MR4439466}, where a similar
parametrization is used. To be accurate, the lines $L_{y}$ only foliate $\mathcal{C} \, \setminus \, \{(0,0,z) : z \in \R\}$. We will abuse notation by
writing $L_{y}(s) = (s,-sy,sy^{2}/2)$ for the
parametrisation of the line $L_{y}$.

\begin{definition}[Light rays]
We say that a line $L$ in $\mathbb{R}^{3}$ is a \emph{light ray} if $L=z + L_y$
for some $z\in \mathbb{R}^3$ and $y\in \mathbb{R}$. In other words, $L$ is a (Euclidean) translate of a line contained in $\mathcal{C}$ (excluding the $t$-axis).
\end{definition}

\begin{remark}
Every light ray can be written as $(0,u,v) + L_y$ for a unique $(u,v) \in \R^{2}$.
\end{remark}

\begin{definition}[Horizontal lines]
A line $\ell$ in $\mathbb{R}^3$ is \emph{horizontal} if it is a
Heisenberg left translate of a horizontal subgroup, that is, there
exists $p\in \mathbb{H}$ and $e\in S^1$ such that $\ell = p \cdot
\mathbb{L}_e$.
\end{definition}

\begin{remark}
Every horizontal line, apart from left translates of the $x$-axis,
can be written as $\ell= \{(as+b,s,(b/2)s+c):\, s\in\mathbb{R}\}$
for a uniquely determined point $(a,b,c)\in \mathbb{R}^3$.
\end{remark}

\begin{definition}\label{d:duality} We define the following correspondence between
points and lines:
\begin{itemize}
\item To a point $p=(x,y,t)\in \mathbb{H}$, we associate the
\emph{light ray}
\begin{equation}\label{form27}
\ell^{\ast}(p) = (0,x,t-xy/2)+ L_y \subset (0,x,t-xy/2) +
\mathcal{C} \subset \mathbb{R}^3.
\end{equation}
(This formula will be motivated by Lemma \ref{l:IncidenceDuality} below.)
\item To a point $p^{\ast}=(a,b,c)\in \mathbb{R}^3$, we associate
the \emph{horizontal line}
\begin{displaymath}
\ell(p^{\ast}) = \{(as+b,s,\tfrac{b}{2}s+c):\, s\in \mathbb{R}\}.
\end{displaymath}
\end{itemize}
Given a set $\mathcal{P}$ of points in $\mathbb{H}$, we define the family of light rays
\begin{equation}\label{form1}
\ell^{\ast}(\mathcal{P}) = \bigcup_{p\in \mathcal{P}}
\ell^{\ast}(p).
\end{equation}
\end{definition}

\begin{remark} It is worth observing that the point $(0,x,t-xy/2)$ appearing in formula \eqref{form27} is nearly the vertical projection of $(x,y,t)$ to the $xt$-plane; the actual formula for this projection would be $\pi_{xt}(x,y,t) = (x,0,t - xy/2)$. It follows from this observation that
\begin{equation}\label{form28} \ell^{\ast}((u,0,v) \cdot (0,y,0)) = (0,u,v) + L_{y}, \qquad u,v,y \in \R, \end{equation}
because $\pi_{xt}((u,0,v) \cdot (0,y,0)) = (u,0,v)$.
\end{remark}

Under the point-line correspondence in Definition \ref{d:duality},
incidences between points and horizontal lines in $\mathbb{H}$ are
in one-to-one correspondence with incidences between light rays
and points in $\mathbb{R}^3$.

\begin{lemma}[Incidences are preserved under
duality]\label{l:IncidenceDuality} For $p \in \He$ and $p^{\ast} \in \R^{3}$, we have
\begin{displaymath}
p \in \ell(p^{\ast}) \quad \Longleftrightarrow \quad  p^{\ast} \in
\ell^{\ast}(p).
\end{displaymath}
\end{lemma}

\begin{proof} Let $p=(x,y,t) \in \He$ and $p^{\ast}=(a,b,c) \in \R^{3}$. The condition $p \in \ell(p^{\ast})$ is equivalent to
\begin{displaymath}
\left\{ \begin{array}{l}ay+b=x\\\tfrac{b}{2} y + c=t.\end{array}
\right.
\end{displaymath}
Recalling the notation $L_{y}(s) = (s,-sy,sy^{2}/2)$, this is further equivalent to
\begin{equation}\label{form42} p^{\ast} = (a,b,c) = (0,x,t - xy/2) + L_{y}(a). \end{equation}
Finally, \eqref{form42} is equivalent to $p^{\ast} \in \ell^{\ast}(p)$. \end{proof}

\subsection{Measures on the space of horizontal lines}\label{s:measures} The duality $p \mapsto \ell(p)$ between points in $p \in \R^{3}$ and horizontal lines $\ell(p)$ in Definition \ref{d:duality} allows one to push-forward Lebesgue measure "$\mathrm{Leb}$" on $\R^{3}$ to construct a measure "$\mathfrak{m}$" on the set of horizontal lines:
\begin{displaymath} \mathfrak{m}(\mathcal{L}) := (\ell_{\sharp}\mathrm{Leb})(\mathcal{L}) = \mathrm{Leb}(\{p \in \R^{3} : \ell(p) \in \mathcal{L}\}). \end{displaymath}
There is, however, a more commonly used measure on the space of
horizontal lines. This measure "$\mathfrak{h}$" is discussed
extensively for example in \cite[Section 2.3]{MR4127898}. The
measure $\mathfrak{h}$ is the unique (up to a multiplicative
constant) non-zero left invariant measure on the set of horizontal
lines. One possible formula for it is the following:
\begin{equation}\label{form20} \mathfrak{h}(\mathcal{L}) = \int_{S^{1}} \mathcal{H}^{3}(\{w \in \W_{e} : \pi_{e}^{-1}\{w\} \in \mathcal{L}\}) \, d\mathcal{H}^{1}(e). \end{equation}
Let $f \in L^{1}(\He)$, and consider the weighted measure $\mu_{f}
:= f \, d\mathrm{Leb}$. Then, starting from the definition
\eqref{form20}, it is easy to check that
\begin{equation}\label{form21} \int_{S^{1}} \|\pi_{e}\mu_{f}\|_{L^{2}}^{2} \, d\mathcal{H}^{1}(e) = \int Xf(\ell)^{2} \, d\mathfrak{h}(\ell), \end{equation}
where $Xf(\ell) := \int_{\ell} f \, d\mathcal{H}^{1}$.

While the measure $\mathfrak{h}$ is mutually absolutely continuous
with respect to $\mathfrak{m}$, the Radon-Nikodym derivative is
not bounded (from above and below): with our current notational
conventions, the lines $\ell(p)$ are never parallel to the
$x$-axis, and the $\mathfrak{m}$-density of lines making a small
angle with the $x$-axis is smaller than their
$\mathfrak{h}$-density. The problem can be removed by restricting
our considerations to lines which make a substantial angle with
the $x$-axis. For example, let $\mathcal{L}_{\angle}$ be the set
of horizontal lines which have slope at most $1$ relative to the
$y$-axis; thus
\begin{displaymath} \mathcal{L}_{\angle} = \ell(\{(a,b,c) \in \R^{3} : |a| \leq 1\}). \end{displaymath}
Then, $\mathfrak{m}(\mathcal{L}) \sim \mathfrak{h}(\mathcal{L})$
for all Borel sets $\mathcal{L} \subset \mathcal{L}_{\angle}$. The
lines in $\mathcal{L}_{\angle}$ coincide with pre-images of the
form $\pi_{e}^{-1}\{w\}$, $e \in S \subset S^{1}$, where $S$
consists of those vectors making an angle at most $45^{\circ}$
with the $y$-axis. Now, \eqref{form21} also holds in the following
restricted form:
\begin{equation}\label{form22} \int_{S} \|\pi_{e}\mu_{f}\|_{L^{2}}^{2} d\mathcal{H}^{1}(e) = \int_{\mathcal{L}_{\angle}} Xf(\ell)^{2} \, d\mathfrak{h}(\ell) \sim \int_{\mathcal{L}_{\angle}} Xf(\ell)^{2} \, d\mathfrak{m}(\ell). \end{equation}
This equation will be useful in establishing Theorem
\ref{mainTechnical}. This will, formally, only prove Theorem
\ref{mainTechnical} with "$S$" in place of "$S^{1}$", but the
original version is easy to deduce from this apparently weaker
version.

\subsection{Ball-plate duality}\label{s:ballPlate} Recall from \eqref{form1} the definition of the (dual) line set $\ell^{\ast}(P)$ for $P \subset \He$. What does $\ell^{\ast}(B_{\He}(p,r))$ look like? The answer is: a \emph{plate} tangent to the cone $\mathcal{C}$. Informally speaking, for $r \in (0,\tfrac{1}{2}]$, an $r$-plate tangent to $\mathcal{C}$ is a rectangle of dimensions $\sim (1 \times r \times r^{2})$ whose long side is parallel to a light ray, and whose orientation is such that the plate is roughly tangent to $\mathcal{C}$, see Figure \ref{fig1}. To prove rigorously that $\ell^{\ast}(B_{\He}(p,r))$ looks like such a plate (inside $B(1)$), we need to be more precise with the definitions.

Recall that the cone $\mathcal{C}$ is a rotation of the "standard"
cone $\mathcal{C}_{0} = \{(x,y,z) : z^{2} = x^{2} + y^{2}\}$.
\begin{figure}[h!]
\begin{center}
\begin{overpic}[scale = 1]{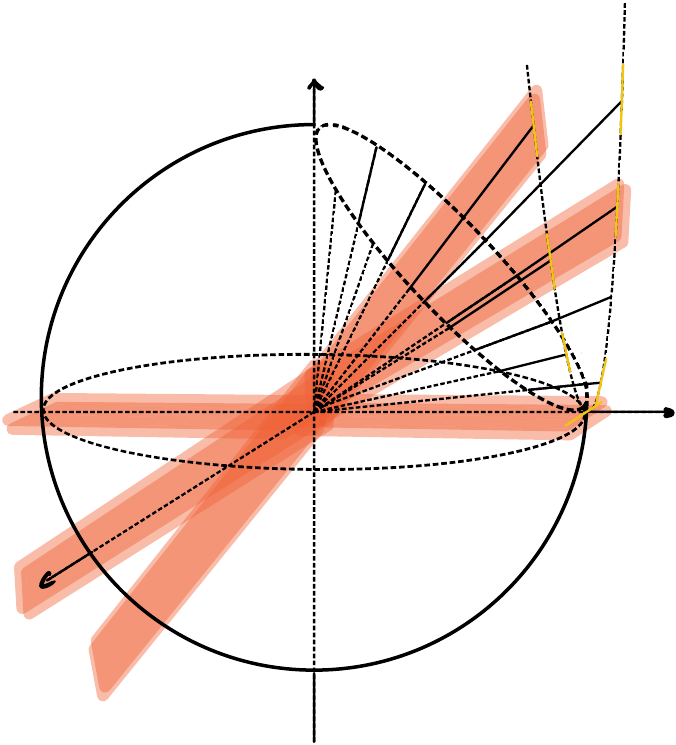}
\put(90,40){$x$} \put(-1,20){$y$} \put(45,88){$z$}
\put(86,70){$\mathbb{P} = \{(1,-y,\tfrac{y^{2}}{2}) : y \in \R\}$}
\put(51,72){$\mathcal{C}$}
\end{overpic}
\caption{The cone $\mathcal{C}$, the parabola $\mathbb{P}$, and
three $r$-plates.}\label{fig1}
\end{center}
\end{figure}
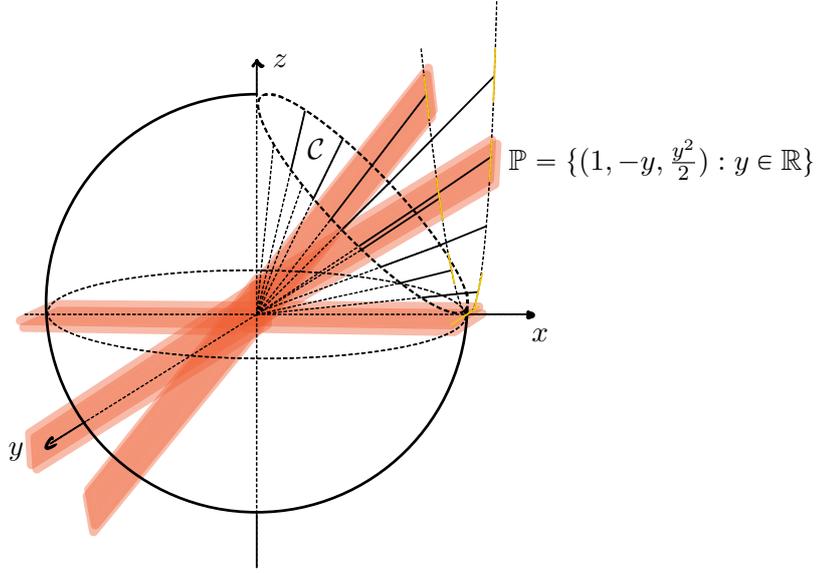

The intersection of $\mathcal{C}$ with the plane $\{x = 1\}$ is
the parabola
\begin{displaymath} \mathbb{P} = \{(1,-y,y^{2}/2) : y \in \R\}. \end{displaymath}
For every $r \in (0,\tfrac{1}{2}]$ and $p \in \mathbb{P}$, choose
a rectangle $\mathcal{R} = \mathcal{R}_{r}(p)$ of dimensions $r
\times r^{2}$ in the plane $\{x = 1\}$, centred at $p$, such that
the longer $r$-side is parallel to the tangent line of
$\mathbb{P}$ at $p$. Then $\mathbb{P} \cap B(0,cr) \subset
\mathcal{R}$ for an absolute constant $c > 0$. Now, the
\emph{$r$-plate centred at $p$} is the set obtained by sliding the
rectangle $\mathcal{R}$ along the light ray containing $p$ inside
$\{|x| \leq 1\}$, see Figure \ref{fig1}. We make this even more
formal in the next definition.

\begin{definition}[$r$-plate]\label{def:plate} Let $r \in (0,\tfrac{1}{2}]$, and let $p = (1,-y,y^{2}/2) \in \mathbb{P} \subset \mathcal{C}$ with $y \in [-1,1]$. Let $\mathcal{R}_{r}(0) := [-r,r] \times [-r^{2},r^{2}]$, and define $\mathcal{R}_{r}(y) := M_{y}(\mathcal{R}_{r}(0)) \subset \R^{2}$, where
\begin{displaymath} M_{y} = \begin{pmatrix} 1 & 0 \\ -y & 1 \end{pmatrix} \end{displaymath}
(The rectangle $\mathcal{R}_{r}(y)$ is the intersection of an
$r$-plate with the plane $\{x = 0\}$.) Define
\begin{displaymath} \mathcal{P}_{r}(p) := \{(0,\vec{r}) + L_{y}([-1,1]) : \vec{r} \in \mathcal{R}_{r}(y)\}, \end{displaymath}
The set $\mathcal{P}_{r}(p)$ is called the \emph{$r$-plate centred
at $p \in \mathbb{P}$.} In general, an \emph{$r$-plate} is any
translate of one of the sets $\mathcal{P}_{r}(p)$, for $p =
(1,-y,y^{2}/2)$ with $y \in [-1,1]$, and $r \in (0,\tfrac{1}{2}]$.

For the $r$-plate $\mathcal{P}_{r}(p)$, we also commonly use the notation $\mathcal{P}_{r}(y)$, where $p = (1,-y,y^{2}/2)$.

\end{definition}

\begin{remark} Since we require $y \in [-1,1]$ in Definition \ref{def:plate}, it is clear that an $r$-plate contains, and is contained in, a rectangle of dimensions $\sim (1 \times r \times r^{2})$. It is instructive to note that the number of "essentially distinct" $r$-plates intersecting $B(0,1)$ is roughly $r^{-4}$: to see this, take a maximal $r$-separated subset of $\mathbb{P}_{r} \subset \mathbb{P}$, and note that for each $p \in \mathbb{P}_{r}$, the plate $\mathcal{P}_{r}(p)$ has volume $r^{3}$. Therefore it takes $\sim r^{-3}$ translates of $\mathcal{P}_{r}(p)$ to cover $B(0,1)$. This $r^{-4}$-numerology already suggests that the various $r$-plates might correspond to Heisenberg $r$-balls via duality. \end{remark}

To relate the plates $\mathcal{P}_{r}$ to Heisenberg balls, we
define a slight modification of the plates $\mathcal{P}_{r}$.
Whereas $\mathcal{P}_{r}$ is a union of (truncated) light rays in
one fixed direction, the following "modified" plates contain full
light rays in an $r$-arc of directions. These "modified" plates
will finally match the duals of Heisenberg balls, see
Proposition \ref{prop1}.

\begin{definition}[Modified $r$-plate]\label{def:modPlate} Let $r \in (0,\tfrac{1}{2}]$ and $y \in [-1,1]$. Let $\mathcal{R}_{r}(y_{0}) \subset \R^{2}$ be the rectangle from Definition \ref{def:plate}. For $(u,v) \in \R^{2}$, define the \emph{modified $r$-plate}
\begin{equation}\label{form3} \Pi_{r}(u,v,y) := (0,u,v) + \{(0,\vec{r}) + L_{y'} : \vec{r} \in \mathcal{R}_{r}(y) \text{ and } |y' - y| \leq r\}. \end{equation}
\end{definition}

\begin{remark}\label{rem1} The relation between the sets $\mathcal{P}_{r}$ and $\Pi_{r}$ is that the following holds for some absolute constant $c > 0$: if $r \in (0,\tfrac{1}{2}]$, $y \in [-1,1]$, and $u,v \in \R$, then
\begin{equation}\label{form9} \Pi_{cr}(u,v,y) \cap \{(s,y,z) : |s| \leq 2\} \subset (0,u,v) + \mathcal{P}_{r}(y) \subset \Pi_{r}(u,v,y). \end{equation}
(The constant "$2$" is arbitrary, but happens to be the one we need.) To see this, it suffices to check the case $u = 0 = v$. Consider
the "slices" of $\Pi_{r}(0,0,y)$ and $\mathcal{P}_{r}(y)$ with a
fixed plane $\{x = s\}$ for $|s| \leq 1$. If $s = 0$, both slices
coincide with the rectangle $\mathcal{R}_{y}(y)$. If $0 < |s| \leq
1$, the slice $\Pi_{r}(0,0,y) \cap \{x = s\}$ can be written as a
sum
\begin{displaymath} \Pi_{r}(0,0,y) \cap \{x = s\} = \mathcal{R}_{r}(y) + \{L_{y'}(s) : |y - y'| \leq r\}, \end{displaymath}
whereas $\mathcal{P}_{r}(y) \cap \{x = s\} = \mathcal{R}_{r}(y) +
\{L_{y}(s)\}$. The relationship between these two slices is
depicted in Figure \ref{fig2}.
\begin{figure}[h!]
\begin{center}
\begin{overpic}[scale = 0.9]{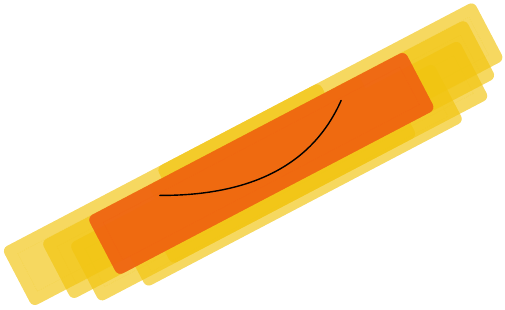}
\end{overpic}
\caption{The red box is the slice $\mathcal{P}_{r}(y) \cap \{x =
s\}$. The slice $\Pi_{r}(0,0,y) \cap \{x = s\}$ is a union of the
yellow boxes centred along the black curve $\{L_{y'}(s) : |y' - y|
\leq r\}$. All the boxes individually are translates of
$\mathcal{R}_{r}(y)$.}\label{fig2}
\end{center}
\end{figure}
After this, we leave it to the reader to verify that
$\Pi_{cr}(0,0,y) \cap \{x = s\} \subset \mathcal{P}_{r}(y) \cap
\{x = s\}$ if $c > 0$ is sufficiently small, and for $|s| \leq 2$.

We record the following consequence of \eqref{form9}:
$\Pi_{r}(u,v,y) \cap \{(s,y,z) : |s| \leq 1\}$ is
contained in a tube of width $r$ around the line $(0,u,v) +
L_{y}$. This is because $\mathcal{P}_{r}(y)$ is obviously
contained in a tube of width $\sim r$ around $L_{y}$ (this is a
very non-sharp statement, using only that the longer side of
$\mathcal{R}_{y}(r)$ has length $r$.)
\end{remark}

We then show that the $\ell^{\ast}$-duals of Heisenberg balls are essentially modified plates:

\begin{proposition}\label{prop1} Let $p  = (u_{0},0,v_{0}) \cdot (0,y_{0},0)$, $r \in (0,\tfrac{1}{2}]$, and $B := B_{\He}(p,r)$. Then,
\begin{equation}\label{form41} \ell^{\ast}(B) \subset \Pi_{2r}(u_{0},v_{0},y_{0}) \subset \ell^{\ast}(CB), \end{equation}
where $C > 0$ is an absolute constant, and $CB = B_{\He}(p,Cr)$.
\end{proposition}

\begin{remark}\label{rem2} To build a geometric intuition, it will be helpful to notice the following. The $y$-coordinate of the point $p = (u_{0},0,v_{0}) \cdot (0,y_{0},0) = (u_{0},y_{0},v_{0} + \tfrac{1}{2}u_{0}y_{0})$ is "$y_{0}$".  On the other hand, while the modified plate $\Pi_{2r}(u_{0},v_{0},y_{0})$ contains many lines, they are all "close" to the "central" line $(0,u_{0},v_{0}) + L_{y_{0}}$ (see Definition \ref{form3}). According to the inclusions in \eqref{form41}, this means that the "direction" $L_{y_{0}}$ of the modified plate containing the dual $\ell^{\ast}(B(p,r))$ is determined by the $y$-coordinate of $p$. Even less formally: \emph{Heisenberg balls whose centres have the same $y$-coordinate are dual to parallel plates}.

\end{remark}

\begin{proof}[Proof of Proposition \ref{prop1}] To prove the inclusion $\ell^{\ast}(B) \subset \Pi_{2r}(u_{0},v_{0},y_{0})$, let $q \in B_{\He}(p,r)$, and write $q := (u,0,v) \cdot (0,y,0)$ with $(u,v) \in \R^{2}$ and $y \in \R$. First, we note that
\begin{equation}\label{form4} |y - y_{0}| \leq d_{\He}(p,q) \leq r.  \end{equation}
Let $\pi_{xt}$ be the vertical projection to the $xt$-plane $\{(u',0,v') : u',v' \in \R\}$. Then $(u,0,v) = \pi_{xt}(q) \in \pi_{xt}(B)$ by the definition of $\pi_{xt}$. We now observe that $B =
(u_{0},0,v_{0}) \cdot B_{\He}((0,y_{0},0),r)$, so
\begin{displaymath} \pi_{xt}(B ) = (u_{0},0,v_{0}) + \pi_{xt}(B_{\He}((0,y_{0},0),r)). \end{displaymath}
We claim that
\begin{equation}\label{form2} \pi_{xt}(B_{\He}((0,y_{0},0),r)) \subset \{(u',0,v') : (u',v') \in \mathcal{R}_{2r}(y_{0})\}. \end{equation}
This will prove that
\begin{equation}\label{form5} (u,0,v) \in (u_{0},0,v_{0}) + \{(u',0,v') : (u',v') \in \mathcal{R}_{2r}(y_{0})\}. \end{equation}
Recalling the definition \eqref{form3}, a combination of
\eqref{form4} and \eqref{form5} now shows that
\begin{displaymath} \ell^{\ast}(q) = \ell^{\ast}((u,0,v) \cdot (0,y,0)) \stackrel{\eqref{form28}}{=} (0,u,v) + L_{y} \subset \Pi_{2r}(u_{0},v_{0},y_{0}). \end{displaymath}
This will complete the proof of the inclusion $\ell^{\ast}(B)
\subset \Pi_{2r}(u_{0},v_{0},y_{0})$.

Let us then prove \eqref{form2}. Pick $(x,y,t) \in
B_{\He}((0,y_{0},0),r)$. Then,
\begin{displaymath} \|(x,y - y_{0},t + \tfrac{1}{2}xy_{0})\| = d_{\He}((x,y,t),(0,y_{0},0)) \leq r, \end{displaymath}
so
\begin{equation}\label{form6} |x| \leq r, \quad |y - y_{0}| \leq r, \quad \text{and} \quad |t + \tfrac{1}{2}xy_{0}| \leq r^{2}. \end{equation}
Now, to prove \eqref{form2}, recall that $\pi_{xt}(x,y,t) = (x,0,t
- \tfrac{1}{2}xy)$. Thus, we need to show that $(x,t -
\tfrac{1}{2}xy) \in \mathcal{R}_{2r}(y_{0}) =
M_{y_{0}}(\mathcal{R}_{2r}(0))$. Equivalently, $M_{y_{0}}^{-1}(x,t
- \tfrac{1}{2}xy) \in \mathcal{R}_{2r}(0)$. Recalling the
definition of $M_{y}$, one checks that
\begin{align*} M_{y_{0}}^{-1}(x,t - \tfrac{1}{2}xy) & = \begin{pmatrix} 1 & 0 \\ y_{0} & 1 \end{pmatrix}(x,t - \tfrac{1}{2}xy)\\
& = (x,xy_{0} + t - \tfrac{1}{2}xy)\\
& = (x,t + \tfrac{1}{2}xy_{0} + \tfrac{1}{2}x(y_{0} - y)).
\end{align*} Using \eqref{form6}, we finally note that the point
on the right lies in the parabolic rectangle
$\mathcal{R}_{2r}(0)$. This concludes the proof of \eqref{form2}.

Let us then prove the inclusion $\Pi_{r}(u_{0},v_{0},y_{0})
\subset \ell^{\ast}(CB)$. The set $\Pi_{r}(u_{0},v_{0},y_{0})$ is
a union of the lines $(0,u_{0},v_{0}) + (0,\vec{r}) + L_{y}$,
where $\vec{r} \in \mathcal{R}_{r}(y_{0})$ and $|y - y_{0}| \leq
r$. We need to show that every such line can be realised as
$\ell^{\ast}(q)$ for some $q \in B_{\He}(p,Cr)$. In this task, we
are aided by the formula
\begin{displaymath} \ell^{\ast}((u,0,v) \cdot (0,y,0)) = (0,u,v) + L_{y} \end{displaymath}
observed in \eqref{form27}. This formula shows that we \emph{need}
to define $q := (u,0,v) \cdot (0,y,0)$, where $(u,v) :=
(u_{0},v_{0}) + \vec{r}$, and $y$ is as in "$L_{y}$". Then we just
have to hope that $q \in B_{\He}(p,Cr)$.

Recalling that $p = (u_{0},0,v_{0}) \cdot (0,y_{0},0)$, one can
check by direct computation that
\begin{equation}\label{form28a} d_{\He}(p,q) = \|(u_{0} - u,y_{0} - y,v_{0} - v + y_{0}(u_{0} - u) + \tfrac{1}{2}(u - u_{0})(y_{0} - y)\|. \end{equation}
On the other hand, one may easily check that $(u,v) \in
(u_{0},v_{0}) + \mathcal{R}_{r}(y_{0})$ is equivalent to
\begin{displaymath} (u - u_{0},v - v_{0} + y_{0}(u - u_{0})) \in \mathcal{R}_{r}(0), \end{displaymath}
which implies $|u - u_{0}| \leq r$ and $|v - v_{0} + y_{0}(u -
u_{0})| \leq r^{2}$. Since moreover $|y - y_{0}| \leq r$ by
assumption, it follows from \eqref{form28a} and the definition of
the norm $\|\cdot\|$ that $d_{\He}(p,q) \lesssim r$. This
completes the proof. \end{proof}

We close the section with two additional auxiliary results:

\begin{proposition}\label{prop2} Let $p,q \in \He$ and $r \in (0,\tfrac{1}{2}]$, and assume that $\|p\| \leq 1/10$. Assume moreover that $\ell^{\ast}(p) \cap B(1) \subset \ell^{\ast}(B_{\He}(q,r))$. Then $p \in B_{\He}(q,Cr)$ for some absolute constant $C > 0$.
\end{proposition}

\begin{proof} Write $p = (u,0,v) \cdot (0,y,0)$, so that $\ell^{\ast}(p) = (0,u,v) + L_{y}$. Since $\|p\| \leq 1/10$, in particular $|u| + |v| \leq 1/5$. By the previous proposition, we already know that
\begin{displaymath} [(0,u,v) + L_{y}] \cap B(1) = \ell^{\ast}(p) \cap B(1) \subset \Pi_{2r}(u_{0},v_{0},y_{0}), \end{displaymath}
where we have written $q = (u_{0},0,v_{0}) \cdot (0,y_{0},0)$.
Since $(0,u,v) \in B(1)$, we know that $(0,u,v) \in
\ell^{\ast}(p)\cap \Pi_{2r}(u_{0},v_{0},y_{0})$. But
\begin{displaymath} \Pi_{2r}(u_{0},v_{0},y_{0}) \cap \{x = 0\} = \{(0,u',v') : (u',v') \in (u_{0},v_{0}) + \mathcal{R}_{y_{0}}(r)\}, \end{displaymath}
so we may deduce that
\begin{equation}\label{form7} (u,v) \in (u_{0},v_{0}) + \mathcal{R}_{y_{0}}(r). \end{equation}
Moreover, in Remark \ref{rem1} we noted that
$\Pi_{2r}(u_{0},v_{0},y_{0}) \cap B(1)$ is contained in the $\sim
r$-neighbourhood $T$ of the line $(0,u_{0},v_{0}) + L_{y_{0}}$.
Therefore also $(0,u,v) + L_{y} \cap B(1) \subset T$. This implies
that $\angle(L_{y},L_{y_{0}}) \lesssim r$, and hence $|y - y_{0}|
\lesssim r$.

Now, we want to use \eqref{form7} and $|y - y_{0}| \lesssim r$ to
deduce that $d_{\He}(p,q) \lesssim r$. We first expand
\begin{equation}\label{form8} d_{\He}(p,q) = \|(u_{0} - u,y_{0} - y,v_{0} - v + y_{0}(u_{0} - u) + \tfrac{1}{2}(u - u_{0})(y_{0} - y)\|. \end{equation}
Then, using the definition of $\mathcal{R}_{y_{0}}(r) =
M_{y}(\mathcal{R}_{0}(r))$, we note that \eqref{form7} is
equivalent to
\begin{displaymath} (u - u_{0},v - v_{0} + y_{0}(u - u_{0})) \in \mathcal{R}_{r}(0). \end{displaymath}
Combined with $|y - y_{0}| \lesssim r$, and recalling the
definition of $\|\cdot\|$, this shows that the right hand side of
\eqref{form8} is bounded by $\lesssim r$, as claimed. \end{proof}

We already noted in Remark \ref{rem2} that the (modified) $2r$-plates containing $\ell^{\ast}(B(p_{1},r))$ and $\ell^{\ast}(B(p_{2},r))$ have (almost) the same direction if the points $p_{1},p_{2}$ have (almost) the same $y$-coordinate. In this case, if $d_{\He}(p_{1},p_{2}) \geq Cr$, it is natural to expect that $\ell^{\ast}(B(p_{1},r))$ and $\ell^{\ast}(B(p_{2},r))$ are disjoint, at least inside $B(1)$. The next lemma verifies this intuition.

\begin{lemma}\label{lemma2} Let $p_{1} = (u_{1},0,v_{1}) \cdot (0,y_{1},0) \in B_{\He}(1)$ and $p_{2} = (u_{2},0,v_{2}) \cdot (0,y_{2},0) \in B_{\He}(1)$ be points with the properties
\begin{equation}\label{form37} |y_{1} - y_{2}| \leq r \quad \text{and} \quad \ell^{\ast}(B_{\He}(p_{1},r)) \cap \ell^{\ast}(B_{\He}(p_{2},r)) \cap B(1) \neq \emptyset. \end{equation}
Then, $d_{\He}(p_{1},p_{2}) \lesssim r$. \end{lemma}

\begin{proof} We may reduce to the case $y_{1} = y_{2}$ by the following argument. Start by choosing a point $p_{2}' \in B_{\He}(p_{2},r)$ such that the $y$-coordinate of $p_{2}'$ equals $y_{1}$. This is possible, because $|y_{1} - y_{2}| \leq r$, and the projection of $B_{\He}(p_{2},r)$ to the $xy$-plane is a Euclidean disc of radius $r$. Then, notice that $B_{\He}(p_{2},r) \subset B_{\He}(p_{2}',2r)$, so
\begin{displaymath} \ell^{\ast}(B_{\He}(p_{1},2r)) \cap \ell^{\ast}(B_{\He}(p_{2}',2r)) \cap B(1) \neq \emptyset. \end{displaymath}
Now, if we have already proven the lemma in the case $y_{1} = y_{2}$ (and for "$2r$" in place of "$r$"), it follows that $d_{\He}(p_{1},p_{2}') \lesssim r$, and finally $d_{\He}(p_{1},p_{2}) \leq d_{\He}(p_{1},p_{2}') + d_{\He}(p_{2}',p_{2}) \lesssim r$.

Let us then assume that $y_{1} = y_{2} = y$. It follows from \eqref{form37} and the first inclusion in Proposition \ref{prop1} combined with the first inclusion in \eqref{form9} that
\begin{displaymath} ((0,u_{1},v_{1}) + \mathcal{P}_{Cr}(y)) \cap ((0,u_{2},v_{2}) + \mathcal{P}_{Cr}(y)) \neq \emptyset \end{displaymath}
for some absolute constant $C > 0$. Let "$x$" be a point in the intersection, and (using the definition of $\mathcal{P}_{Cr}(y)$), express $x$ in the two following ways:
\begin{displaymath} (0,u_{1},v_{1}) + (0,\vec{r}_{1}) + L_{y}(s) = x = (0,u_{2},v_{2}) + (0,\vec{r}_{2}) + L_{y}(s), \end{displaymath}
where $\vec{r_{1}} \in \mathcal{R}_{Cr}(y) = M_{y}(\mathcal{R}_{Cr}(0))$ and $\vec{r_{2}} \in M_{y}(\mathcal{R}_{r}(0))$, and $s \in [-1,1]$. The terms $L_{y}(s)$ conveniently cancel out, and we find that
\begin{displaymath} (u_{1},v_{1}) - (u_{2},v_{2}) = \vec{r_{2}} - \vec{r_{1}} \in M_{y}(\mathcal{R}_{2Cr}(0)), \end{displaymath}
or equivalently
\begin{equation}\label{form39} (u_{1} - u_{2},v_{1} - v_{2} + y(u_{1} - u_{2})) = M_{y}^{-1}(u_{1} - u_{2},v_{1} - v_{2}) \in \mathcal{R}_{2Cr}(0). \end{equation}
We have already computed in \eqref{form8} that
\begin{displaymath} d_{\He}(p_{1},p_{2}) = \|(u_{1} - u_{2},0,v_{1} - v_{2} + y(u_{1} - u_{2})\|, \end{displaymath}
and now it follows immediately from \eqref{form39} that $d_{\He}(p_{1},p_{2}) \lesssim r$. \end{proof}

\section{Discretising Theorem \ref{main}}\label{s:discretisation}

The purpose of this section is to reduce the proof of Theorem
\ref{main} to Theorem \ref{mainTechnical} which concerns
\emph{$(\delta,3)$-sets}. We start by defining these precisely:
\begin{definition}[($\delta,t,C)$-set]\label{deltaTSet} Let $(X,d)$ be a metric space, and let $t \geq 0$ and $C,\delta > 0$. A non-empty bounded set $P \subset X$ is called a \emph{$(\delta,t,C)$-set} if
\begin{displaymath} |P \cap B(x,r)|_{\delta} \leq C r^{t} \cdot |P|_{\delta}, \qquad x \in X, \, r \geq \delta. \end{displaymath}
Here $|A|_{\delta}$ is the smallest number of balls of radius $\delta$ needed to cover $A$.  A family of sets $\mathcal{B}$ (typically: disjoint $\delta$-balls) is called a $(\delta,t,C)$-set if $P := \cup \mathcal{B}$ is a $(\delta,t,C)$-set.
\end{definition}

If $P \subset \He$, or $\mathcal{B} \subset \mathcal{P}(\He)$, the $(\delta,t,C)$-set condition is always tested relative to the metric $d_{\He}$. We then state a $\delta$-discretised version of Theorem \ref{main} for sets of dimension $3$:

\begin{thm}\label{mainTechnical} For every $\eta > 0$, there exists $\epsilon > 0$ and $\delta_{0} > 0$ such that the following holds for all $\delta \in (0,\delta_{0}]$. Let $\mathcal{B}$ be a non-empty $(\delta,3,\delta^{-\epsilon})$-set of $\delta$-balls contained in $B_{\He}(1)$, with $\delta$-separated centres. Let $\mu = \mu_{f}$ be the measure on $\mathbb{H}$ with density
\begin{equation}\label{form34} f := (\delta^{4}|\mathcal{B}|)^{-1} \sum_{B \in \mathcal{B}} \mathbf{1}_{B}. \end{equation}
Then,
\begin{displaymath} \int_{S^{1}} \|\pi_{e}\mu\|_{L^{2}}^{2} \, d\mathcal{H}^{1}(e) \leq \delta^{-\eta}. \end{displaymath}
\end{thm}

The proof of Theorem \ref{mainTechnical} will be given in Section
\ref{s:GWZ}. Deducing Theorem \ref{main} from Theorem
\ref{mainTechnical} involves two steps. The first one, carried out in Section
\ref{s:mainProof}, is to reduce Theorem
\ref{main} to a $\delta$-discretised version, which concerns
$(\delta,t)$-sets with all possible values $t \in [0,3]$. This
statement is Theorem \ref{mainDiscrete} below, a simplified
version of which was stated as Theorem \ref{mainMeasure} in the
introduction.

The second -- and less standard -- step, carried out in this section, is to deduce
Theorem \ref{mainDiscrete} from Theorem \ref{mainTechnical}.
Heuristically, Theorem \ref{mainTechnical} is nothing but the $3$-dimensional case of Theorem \ref{mainDiscrete} -- although in this case the statement looks more quantitative. We therefore need to argue that  \emph{if we already have
Theorem \ref{mainDiscrete} for sets of dimension $3$, then we also have it for sets of dimension
$t \in [0,3]$.} The heuristic is simple: given a set $K
\subset \He$ of dimension $t \in [0,3]$, we start by "adding"
(from the left) to $K$ another -- random -- set $H \subset \He$ of
dimension $3 - t$. Then, we apply the $3$-dimensional version of
Theorem \ref{mainDiscrete} to $H \cdot K$, and this gives the correct
conclusion for $K$. A crucial point is that Theorem \ref{mainDiscrete} concerns the Lebesgue measure (not the dimension) of $\pi_{e}(K)$. This quantity is invariant under left translating $K$. This allows us to control $\mathrm{Leb}(\pi_{e}(H \cdot K))$ in a useful way.

We turn to the details. To deduce Theorem \ref{main} from Theorem \ref{mainTechnical}, we need a corollary of Theorem \ref{mainTechnical}, stated in Corollary \ref{cor1}, which concerns slightly more general measures than ones of the form $\mu = \mu_{f}$ (as in \eqref{form34}):
\begin{definition}[$\delta$-measure] Let $\delta \in (0,1]$ and $C > 0$. A Borel measure $\mu$ on $\He$ is called a \emph{$(\delta,C)$-measure} if $\mu$ has a density with respect to Lebesgue measure, also denoted $\mu$, and the density satisfies
\begin{displaymath} \mu(x) \leq C \cdot \frac{\mu(B_{\He}(x,\delta))}{\mathrm{Leb}(B_{\He}(x,\delta))}, \qquad x \in \He. \end{displaymath}
If the constant $C > 0$ irrelevant, a $(\delta,C)$-measure may also be called a $\delta$-measure. \end{definition}

We will use the following notion of \emph{$\delta$-truncated Riesz energy}:
\begin{equation}\label{rieszEnergy} I_{s}^{\delta}(\mu) := \iint \frac{d\mu(x)\,d\mu(y)}{d_{\He,\delta}(x,y)^{s}}, \end{equation}
where $\mu$ is a Radon measure, $0 \leq s \leq 4$, and $d_{\He,\delta}(x,y) := \max\{d_{\He}(x,y),\delta\}$.

\begin{cor}\label{cor1} For every $\eta > 0$, there exists $\delta_{0},\epsilon_{0} > 0$ such that the following holds for all $\delta \in (0,\delta_{0}]$ and $\epsilon \in (0,\epsilon_{0}]$. Let $\mu$ be a $(\delta,\delta^{-\epsilon})$-probability measure on $B_{\He}(1)$ with $I_{3}^{\delta}(\mu) \leq \delta^{-\epsilon}$. Then, there exists a Borel set $G \subset \He$ such that $\mu(G) \geq 1 - \delta^{\epsilon_{0}}$, and
\begin{equation}\label{form33} \int_{S^{1}} \|\pi_{e}(\mu|_{G})\|_{L^{2}}^{2} \, d\mathcal{H}^{1}(e) \leq \delta^{-\eta}. \end{equation}
\end{cor}

\begin{proof} Fix $\eta > 0$, $\epsilon \in (0,\epsilon_{0}]$, and $\delta \in (0,\delta_{0}]$.  The dependence of $\delta_{0},\epsilon_{0}$ on $\eta$ will eventually be determined by an application of Theorem \ref{mainTechnical}, but we will require at least that $\epsilon_{0} \leq \eta$.

It follows from $I^{\delta}_{3}(\mu) \leq \delta^{-\epsilon}$ and Chebychev's inequality that there exists a set $G_{0} \subset \He$ of measure $\mu(G_{0}) \geq 1 - 3\delta^{\epsilon_{0}}$ such that $\mu(B_{\He}(x,r)) \lesssim \delta^{-\epsilon - \epsilon_{0}}r^{3} \leq \delta^{-2\epsilon_{0}}r^{3}$ for all $x \in G_{0}$ and $r \geq \delta$. Now, for dyadic rationals $0 < \alpha \lesssim \delta^{3 - 2\epsilon_{0}} \leq \delta^{2}$, let
\begin{displaymath} G_{0,\alpha} := \{x \in G_{0} : \tfrac{\alpha}{2} \leq \mu(B_{\He}(x,\delta)) \leq \alpha\}. \end{displaymath}
We discard immediately the sets $G_{0,\alpha}$ with $\alpha \leq \delta^{10}$: the union of these sets has measure $\leq \delta^{5} \leq \delta^{\epsilon_{0}}$ for $\delta > 0$ small enough, so $\mu(G_{1}) \geq 1 - 2\delta^{\epsilon_{0}}$, where
\begin{displaymath} G_{1} := G_{0} \, \setminus \, \bigcup_{\alpha \leq \delta^{10}} G_{0,\alpha}. \end{displaymath}
Now, $G_{1}$ is covered by the sets $G_{0,\alpha}$ with $\delta^{10} \leq \alpha \lesssim \delta^{2}$, and the number of such sets is $m \lesssim \log(1/\delta)$. We let $\{\alpha_{1},\ldots,\alpha_{m}\}$ be an enumeration of these values of "$\alpha$", and we abbreviate $G^{j} := G_{0,\alpha_{j}}$. We note that the union of the sets $G^{j}$ with $\mu(G^{j}) \leq \delta^{2\epsilon_{0}}$ has measure at most $m \cdot \delta^{2\epsilon_{0}} \leq \delta^{\epsilon_{0}}$ (for $\delta > 0$ small), so finally
\begin{displaymath} G := G_{1} \, \setminus \, \bigcup \{G^{j} : 1 \leq j \leq m \text{ and } \mu(G^{j}) \leq \delta^{2\epsilon_{0}}\} \end{displaymath}
has measure $\mu(G) \geq 1 - 2\delta^{\epsilon_{0}} - \delta^{\epsilon_{0}} \geq 1 - \delta^{\epsilon_{0}}$.
Moreover, $G$ is covered by the sets $G^{j}$ with $\mu(G^{j}) \geq \delta^{2\epsilon_{0}}$. Re-indexing if necessary, we now assume that $\mu(G^{j}) \geq \delta^{2\epsilon_{0}}$ for all $1 \leq j \leq m$.

For $1 \leq j \leq m$ fixed, let $\mathcal{B}_{j}$ be a finitely overlapping (Vitali) cover of $G^{j}$ by balls of radius $\delta$, centred at $G^{j}$. Using the facts $G^{j} \subset G_{0}$ and $\mu(G^{j}) \geq \delta^{2\epsilon_{0}}$, and the uniform lower bound $\mu(B_{\He}(x,\delta)) \geq \alpha_{j}/2$ for $x \in G^{j}$, it is easy to check that each $\mathcal{B}_{j}$ is a $(\delta,3,\delta^{-C\epsilon_{0}})$-set with
\begin{equation}\label{form32} |\mathcal{B}_{j}| \lesssim \alpha_{j}^{-1}. \end{equation}
Thus, writing
\begin{displaymath} f_{j} := (\delta^{4}|\mathcal{B}_{j}|^{-1}) \sum_{B \in \mathcal{B}_{j}} \mathbf{1}_{B} \quad \text{and} \quad \mu_{j} := \mu_{f_{j}}, \end{displaymath}
and assuming that $\delta_{0},\epsilon_{0} > 0$ are sufficiently small in terms of $\eta$, we may deduce from Theorem \ref{mainTechnical} that
\begin{displaymath} \int_{S^{1}} \|\pi_{e}(\mu_{j})\|_{L^{2}}^{2} d\mathcal{H}^{1}(e) \leq \delta^{-\eta}, \qquad 1 \leq j \leq m. \end{displaymath}
Finally, it follows from the $(\delta,\delta^{-\epsilon})$-property of $\mu$ that
\begin{displaymath} \mu(x) \lesssim \delta^{-\epsilon} \cdot \frac{\mu(B_{\He}(x,\delta))}{\delta^{4}} \leq \delta^{-\epsilon} \cdot \frac{\alpha_{j}}{\delta^{4}} \stackrel{\eqref{form32}}{\lesssim} \frac{\delta^{-\epsilon}}{\delta^{4}|\mathcal{B}_{j}|} \leq \delta^{-\epsilon} \cdot \mu_{j}(x), \qquad x \in G^{j}. \end{displaymath}
Thus, also the density of $\pi_{e}(\mu|_{G^{j}})$ is bounded from above by the density of $\pi_{e}(\mu_{j})$:
\begin{displaymath} \int_{S^{1}} \|\pi_{e}(\mu|_{G})\|_{L^{2}}^{2} \, d\mathcal{H}^{1}(e) \lesssim \delta^{-\epsilon} \sum_{j = 1}^{m} \int_{S^{1}} \|\pi_{e}(\mu_{j})\|_{L^{2}}^{2} \, d\mathcal{H}^{1}(e) \lesssim \log(1/\delta) \cdot \delta^{-\eta - \epsilon} \leq \delta^{-3\eta}. \end{displaymath}
This completes the proof of \eqref{form33} (with "$3\eta$" in place of "$\eta$"). \end{proof}

The concrete $\delta$-measures we will consider have the form $\eta \ast_{\He} \mu$, where $\mu = \mu_{f}$ has a density of the form \eqref{form34} (these are almost trivially $\delta$-measures), and $\eta$ is a (discrete) probability measure. The notation $\eta \ast_{\He} \mu$ refers to the (non-commutative!) Heisenberg convolution of $\eta$ and $\mu$, that is, the push-forward of $\eta \times \mu$ under the group product $(p,q) \mapsto p \cdot q$. Let us verify that such measures $\eta \ast_{\He} \mu$ are also $\delta$-measures:

\begin{lemma}\label{lemma1} Let $\mu$ be $(\delta,C)$ measure, and let $\eta$ be an arbitrary Borel probability measure on $\He$. Then $\eta \ast_{\He} \mu$ is again a $(\delta,C)$-measure.
\end{lemma}

\begin{proof} Recall that a $(\delta,C)$ measure is absolutely continuous by definition, so the notation "$\mu(p)$" is well-defined for Lebesgue almost every $p \in \He$. The following formulae are valid, and easy to check, for Lebesgue almost every $p \in \He$:
\begin{displaymath} (\eta \ast_{\He} \mu)(p) = \int \mu(q^{-1} \cdot p) \, d\eta(q) \end{displaymath}
and
\begin{equation}\label{form35} \frac{(\eta \ast_{\He} \mu)(B_{\He}(p,r))}{\mathrm{Leb}(B_{\He}(p,r))} = \int \frac{\mu(B_{\He}(q^{-1} \cdot p,r))}{\mathrm{Leb}(B_{\He}(p,r))} \, d\eta(q). \end{equation}
Now, if one applies the $\delta$-measure assumption to the formula on the left hand side, one obtains
\begin{displaymath} (\eta \ast_{\He} \mu)(p) \leq C \int \frac{\mu(B_{\He}(q^{-1} \cdot p,\delta))}{\mathrm{Leb}(B_{\He}(q^{-1} \cdot p,\delta))} \, d\eta(q). \end{displaymath}
Lebesgue measure is invariant under left translations, so
\begin{displaymath} \mathrm{Leb}(B_{\He}(q^{-1} \cdot p,\delta)) = \mathrm{Leb}(B_{\He}(p,\delta)). \end{displaymath}
Therefore, it follows from equation \eqref{form35} that
\begin{displaymath} (\eta \ast_{\He} \mu)(p) \leq C \cdot \frac{(\eta \ast_{\He} \mu)(B_{\He}(p,\delta))}{\mathrm{Leb}(B_{\He}(p,\delta))} \end{displaymath}
for Lebesgue almost every $p \in \He$. This is what we claimed. \end{proof}

We are then ready to state and prove the $\delta$-discretised counterpart of Theorem \ref{main}.

\begin{thm}\label{mainDiscrete} Let $0 \leq s < t \leq 3$. Then, there exist $\epsilon,\delta_{0} > 0$, depending only on $s,t$, such that the following holds for all $\delta \in (0,\delta_{0}]$. Let $\mathcal{B} \neq \emptyset$ be a $(\delta,t,\delta^{-\epsilon})$ set of $\delta$-balls with $\delta$-separated centres, all contained in $B_{\He}(1)$, and let $S \subset S^{1}$ be a Borel set of length $\mathcal{H}^{1}(S) \geq \delta^{\epsilon}$. Then, there exists $e \in S$ such that the following holds: if $\mathcal{B}' \subset \mathcal{B}$ is any sub-family with $|\mathcal{B}'| \geq \delta^{\epsilon}|\mathcal{B}|$, then
\begin{displaymath} \mathrm{Leb}(\pi_{e}(\cup \mathcal{B}')) \geq \delta^{3 - s}. \end{displaymath}
In particular, $\pi_{e}(\cup \mathcal{B}')$ cannot be covered by fewer than $\delta^{-s}$ parabolic balls of radius $\delta$. \end{thm}

\begin{proof} To reach a contradiction, assume that there exists a $(\delta,t,\delta^{-\epsilon})$-set $\mathcal{B}$ of $\delta$-balls with $\delta$-separated centres, contained in $B_{\He}(1)$, and violating the conclusion of Theorem \ref{mainDiscrete}: there exists $s < t$, and for every $e \in S$ (Borel subset of $S^{1}$ of length $\mathcal{H}^{1}(S) \geq \delta^{\epsilon}$), there exists a subset $\mathcal{B}_{e} \subset \mathcal{B}$ with $|\mathcal{B}_{e}| \geq \delta^{\epsilon}|\mathcal{B}|$ with the property
\begin{equation}\label{form36} \mathrm{Leb}(\pi_{e}(\cup \mathcal{B}_{e})) \leq \delta^{3 - s}. \end{equation}
We aim for a contradiction if $\epsilon,\delta$ are sufficiently small. We fix an auxiliary parameter $0 < \eta < (t - s)/2$. Then, we apply Corollary \ref{cor1} to find the constant $\epsilon_{0} > 0$ which depends only on $\eta$. Finally, we will assume, presently, that $\epsilon < \epsilon_{0}/2$, and $\eta + 3\epsilon < t - s$.

Let $\mu$ be the uniformly distributed probability measure on $\cup \mathcal{B}$; in particular $\mu$ is a $\delta$-measure (with absolute constant), and $I_{t}^{\delta}(\mu) \lessapprox \delta^{-\epsilon}$. Apply Proposition \ref{prop:aug} to find a set $H \subset B_{\He}(1)$ of cardinality $|H| \leq \delta^{t - 3}$ such that $I^{\delta}_{3}(\tau \ast_{\He} \mu) \lessapprox \delta^{-\epsilon}$, where $\tau$ is the uniformly distributed probability measure on $H$. Write $\nu := \tau \ast_{\He} \mu$, so $\nu$ is a $\delta$-probability measure by Lemma \ref{lemma1}. Since $\epsilon < \epsilon_{0}/2$ and $I_{3}^{\delta}(\nu) \lessapprox \delta^{-\epsilon}$, it follows form Corollary \ref{cor1} that there exists a set $G \subset \He$ of measure $\nu(G) \geq 1 - \delta^{\epsilon_{0}}$ such that
\begin{equation}\label{form30} \frac{1}{\mathcal{H}^{1}(S)} \int_{S} \|\pi_{e}(\nu|_{G})\|_{L^{2}}^{2} \, d\mathcal{H}^{1}(e) \leq \frac{1}{\mathcal{H}^{1}(S)} \int_{S^{1}} \|\pi_{e}(\nu|_{G})\|_{L^{2}}^{2} \, d\mathcal{H}^{1}(e) \leq \delta^{-\eta - \epsilon}. \end{equation}
Finally, write $B_{e} := H \cdot (\cup \mathcal{B}_{e})$ for all $e \in S^{1}$, and note that $\nu(B_{e}) \geq \delta^{\epsilon}$ for all $e \in S$ (this is a consequence of the general inequality $(\mu_{1} \ast_{\He} \mu_{2})(A \cdot B) \geq (\mu_{1} \times \mu_{2})(A \times B)$). Consequently, also $\nu(G \cap B_{e}) \geq \nu(G) + \nu(B_{e}) - 1 \geq \delta^{\epsilon} - \delta^{\epsilon_{0}} \geq \delta^{\epsilon}/2$, using $\epsilon < \epsilon_{0}/2$. Therefore,
\begin{displaymath} \delta^{2\epsilon}/4 \leq \|\pi_{e}(\nu|_{G \cap B_{e}})\|_{L^{1}}^{2} \leq \mathrm{Leb}(\pi_{e}(B_{e})) \cdot \|\pi_{e}(\nu|_{G})\|_{L^{2}}^{2}, \quad e \in S^{1}, \end{displaymath}
using Cauchy-Schwarz, and it follows from \eqref{form30} that $\mathrm{Leb}(\pi_{e}(B_{e})) \gtrsim \delta^{\eta + 3\epsilon}$ for at least one vector $e \in S$. On the other hand, note that $B_{e} = H \cdot (\cup \mathcal{B}_{e})$ is a union of $\leq \delta^{t - 3}$ left translates of $\cup \mathcal{B}_{e}$, and recall from \eqref{projLeb} that
\begin{displaymath} \mathrm{Leb}(\pi_{e}(p \cdot B)) = \mathrm{Leb}(\pi_{e}(B)), \qquad p \in \He, \, B \subset \He. \end{displaymath}
Therefore, we have the upper bound
\begin{displaymath} \mathrm{Leb}(\pi_{e}(B_{e})) = \mathrm{Leb}(\pi_{e}(H \cdot (\cup \mathcal{B}_{e}))) \stackrel{\eqref{form36}}{\leq} \delta^{t - 3} \cdot \delta^{3 - s} = \delta^{t - s}, \qquad e \in S^{1}. \end{displaymath}
Since $\eta + 3\epsilon < t - s$ by assumption, the previous lower and upper bounds for $\mathrm{Leb}(\pi_{e}(B_{e}))$ are not compatible for $\delta > 0$ small enough. A contradiction has been reached. \end{proof}

\section{Kakeya estimate of Guth, Wang, and Zhang}\label{s:GWZ}

The purpose of this section is to prove Theorem
\ref{mainTechnical}. This will be based on the duality between
horizontal lines and light rays developed in Section
\ref{s:duality}, and an application of a \emph{(reverse) square
function inequality for the cone}, due to Guth, Wang, and Zhang
\cite{MR4151084}. To be precise, we will not need the full power
of this "oscillatory" statement, but rather only a \emph{Kakeya
inequality for plates} in \cite[Lemma 1.4]{MR4151084}. To
introduce the statement, we need to recap some of the terminology
and notation in \cite{MR4151084}. This discussion follows
\cite[Section 1]{MR4151084}, but we prefer a different scaling:
more precisely, in our discussion the geometric objects (plates
and rectangles) of \cite{MR4151084} are dilated by "$R$" on the
frequency side and (consequently) by $R^{-1}$ on the spatial side.

Fix $R \geq 1$, and let
\begin{equation}\label{form10} \Gamma := \Gamma_{R} := \mathcal{C} \cap \{R/2 \leq |\xi| \leq R\}. \end{equation}
Let $\Gamma(1)$ be the $1$-neighbourhood of $\Gamma$, and let $\Theta := \Theta_{R}$ be a finitely overlapping cover of $\Gamma(1)$ by rectangles of dimensions $R \times R^{1/2} \times 1$, whose longest side is parallel to a light ray. The statements in \cite{MR4151084} are not affected by the particular construction of $\Theta$, but in our application, the relevant rectangles are translates of dual rectangles of the $\delta$-plates in Definition \ref{def:plate}, with $\delta = R^{-1/2}$. Indeed, $\delta$-plates are rectangles of dimensions $\sim \delta^{2} \times \delta \times 1$ tangent to $\mathcal{C}$, so their dual rectangles are plates of dimensions $\sim R \times R^{1/2} \times 1$, also tangent to $\mathcal{C}$ (this is because $\mathcal{C}$ has opening angle $\pi/2$, see Figure \ref{fig3}). For concreteness, we will use translated duals of $R^{-1/2}$-plates (as in Definition \ref{def:plate}) to form the collection $\Theta$.

For each $\theta \in \Theta$, let $f_{\theta} \in L^{2}(\R^{3})$ be a function with $\spt \hat{f}_{\theta} \subset \theta$, and consider the square function
\begin{displaymath} Sf := \Big( \sum_{\theta \in \Theta} |f_{\theta}|^{2} \Big)^{1/2}. \end{displaymath}
Then, \cite[Lemma 1.4]{MR4151084} contains an inequality of the following form:
\begin{equation}\label{GWZ} \int_{\R^{3}} |Sf|^{4} \lesssim \sum_{R^{-1/2} \leq s \leq 1} \sum_{d(\tau) = s} \sum_{U \| U_{\tau}} \mathrm{Leb}(U)^{-1}\|S_{U}f\|_{L^{2}}^{4}. \end{equation}
To understand the meaning of the "partial" square functions $S_{U}$ we need to introduce more terminology from \cite{MR4151084}. Fix a dyadic number $s \in [R^{-1/2},1]$ (an "angular" parameter), and write $R' := s^{2}R \in [1,R]$. The $1$-neighbourhood of the truncated cone $\Gamma_{R'} = \mathcal{C} \cap \{|\xi| \sim R'\}$ can be covered by a finitely overlapping family $\Theta_{R'}$ of rectangles of dimensions
\begin{displaymath} R' \times (R')^{1/2} \times 1 = s^{2}R \times sR^{1/2} \times 1. \end{displaymath}
(Here $\Theta_{R}$ agrees with $\Theta$, as defined above.) Consequently, the $(R')^{-1}$-neighbourhood of $\Gamma_{R}$ is covered by the rescaled rectangles
\begin{displaymath} \mathcal{T}_{s} := \{s^{-2}\theta : \theta \in \Theta_{R'}\} \end{displaymath}
of dimensions $R \times s^{-1}R^{1/2} \times s^{-2}$. Note that the family $\mathcal{T}_{1}$ coincides with $\Theta_{R}$ (at least if it is defined appropriately), whereas $\mathcal{T}_{R^{-1/2}}$ consists of $\sim 1$ balls of radius $R$. For every $s \in [R^{-1/2},1]$, the rectangles in $\mathcal{T}_{s}$ are at least as large as those in $\Theta_{R}$, so we may assume that every $\theta \in \Theta_{R}$ is contained in at least one rectangle $\tau \in \mathcal{T}_{s}$. 

For $\theta \in \Theta_{R}$ and $\tau \in \mathcal{T}_{s}$, let $\theta^{\ast}$ and $\tau^{\ast}$ be the dual rectangles of $\theta$ and $\tau$ (here the word "dual" refers to the common notion in Euclidean Fourier analysis, and not the duality in the sense of Proposition \ref{prop1}). Then both $\theta^{\ast}$ and $\tau^{\ast}$ are rectangles centred at the origin, with dimensions
\begin{displaymath} R^{-1} \times R^{-1/2} \times 1 \quad \text{and} \quad R^{-1} \times sR^{-1/2} \times s^{2}, \end{displaymath}
respectively. The longest sides of both $\theta^{\ast}$ and $\tau^{\ast}$ remain parallel to a light ray on $\mathcal{C}$: this is again the convenient property of the "standard" cone $\mathcal{C}$ with opening angle $\pi/2$, see Figure \ref{fig3}. Of course, $\theta^{\ast}$ is an $R^{-1/2}$-plate in the sense of Definition \ref{def:plate}, since the elements $\theta \in \Theta$ were defined as (translates of) duals or $R^{-1/2}$-plates.

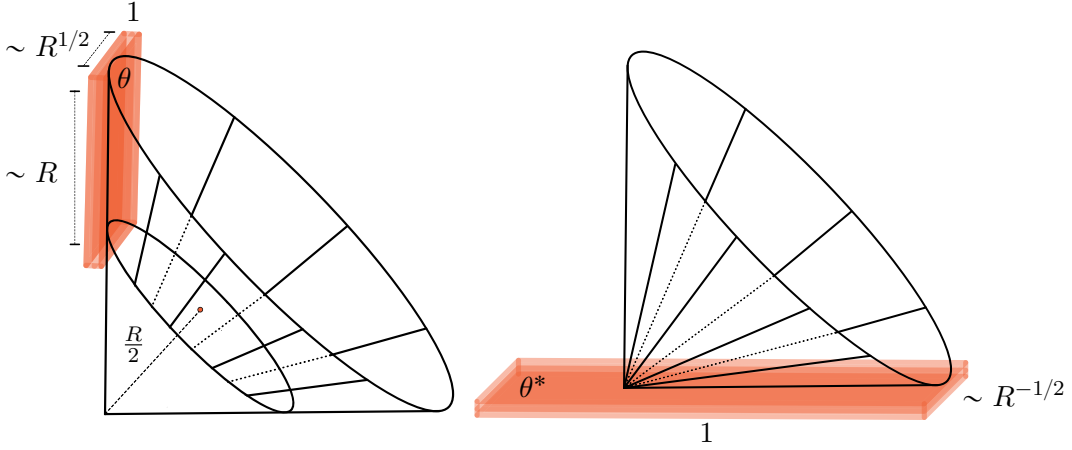
\begin{figure}[h!]
\begin{center}
\begin{overpic}[scale = 0.7]{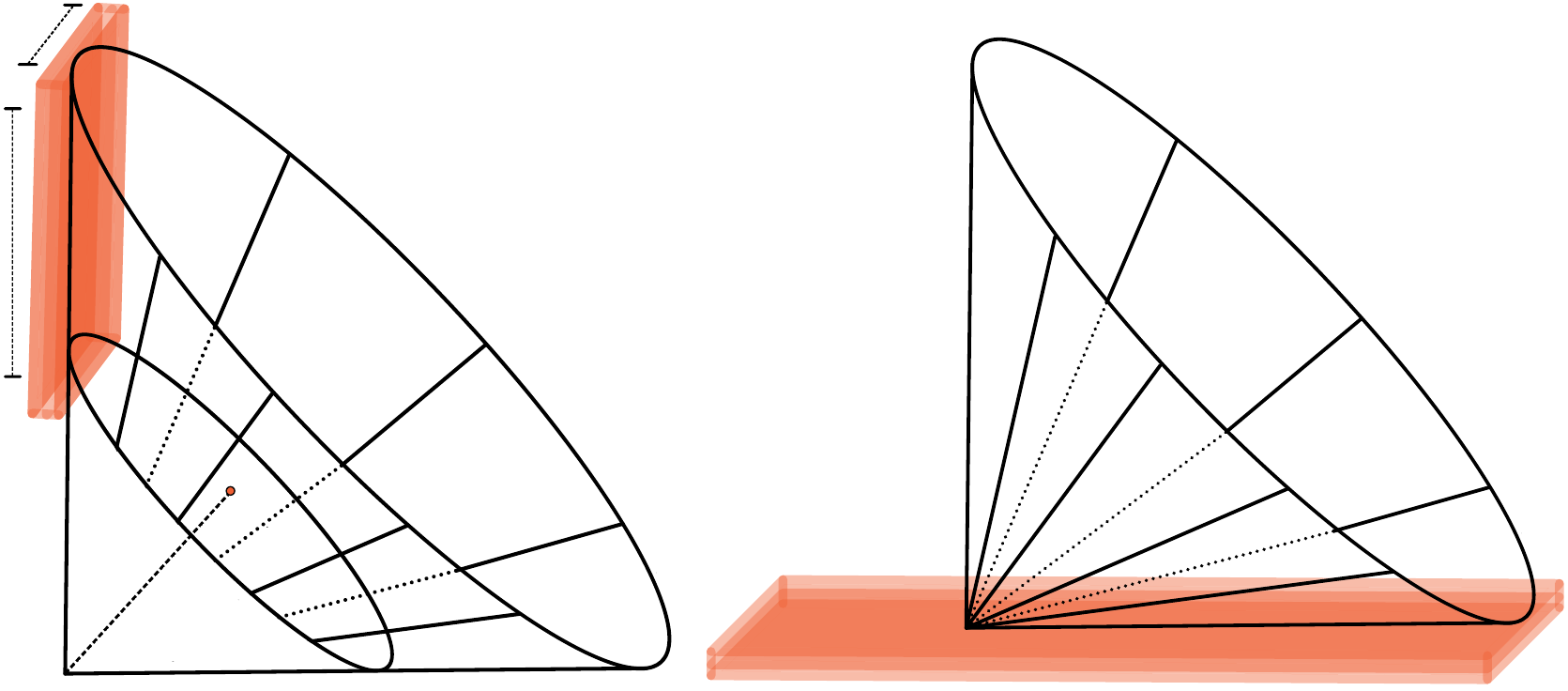}
\put(5.5,37.5){$\theta$}
\put(-7,27){$\sim R$}
\put(-7,41){$\sim R^{1/2}$}
\put(6.5,45){$1$}
\put(6,8){$\tfrac{R}{2}$}
\put(50,3){$\theta^{\ast}$}
\put(70,-2){$1$}
\put(99,2){$\sim R^{-1/2}$}
\end{overpic}
\caption{On the left: the truncated cone $\Gamma$ and one of the plates $\theta$. On the right: the cone $\mathcal{C}$ and the dual plate $\theta^{\ast}$.}\label{fig3}
\end{center}
\end{figure}

The set $\tau^{\ast}$ turns out to be (essentially) a dilate of an $(s^{2}R)^{-1/2}$-plate. For every $\tau \in \mathcal{T}_{s}$, consider $U_{\tau} := s^{-2}\tau^{\ast}$, which is a rectangle of dimensions
\begin{displaymath} s^{-2}R^{-1} \times s^{-1}R^{-1/2} \times 1 = (s^{2}R)^{-1} \times (s^{2}R)^{-1/2} \times 1. \end{displaymath}
In particular, $U_{\tau}$ is an $(s^{2}R)^{-1/2}$-plate, and hence larger than (or at least as large as) $\theta^{\ast}$: if $\theta \subset \tau$, then every translate of $\theta^{\ast}$ is contained in some translate of $10U_{\tau}$. We let $\mathcal{U}_{\tau}$ be a tiling of $\R^{3}$ by rectangles parallel to $U_{\tau}$. Now we may finally define the "partial" square function $S_{U}f$:
\begin{equation}\label{form12} S_{U}f := \Big(\sum_{\theta \subset \tau} |f_{\theta}|^{2} \Big)^{1/2} \cdot \mathbf{1}_{U}, \qquad U \in \mathcal{U}_{\tau}. \end{equation}
We have now explained the meaning of \eqref{GWZ}, except the sum over "$d(\tau) = s$". In our notation, this means the same as summing over $\tau \in \mathcal{T}_{s}$.

We are then prepared to prove Theorem \ref{mainTechnical}.

\begin{proof}[Proof of Theorem \ref{mainTechnical}] Let $\delta \in (0,\tfrac{1}{2}]$, and let $\mathcal{B}$ be a $(\delta,3,\delta^{-\epsilon})$-set of $\delta$-balls with $\delta$-separated centres. In the statement of Theorem \ref{mainTechnical}, it was assumed that $\cup \mathcal{B} \subset B_{\He}(1)$, but for slight technical convenience we strengthen this (with no loss of generality) to $\cup \mathcal{B} \subset B_{\He}(c)$ for a small absolute constant $c > 0$. As in the statement of Theorem \ref{mainTechnical}, let $\mu$ be the measure on $\He$ with density
\begin{displaymath} f := (\delta^{4}|\mathcal{B}|)^{-1} \sum_{B \in \mathcal{B}} \mathbf{1}_{B}. \end{displaymath}
Following the discussion Section \ref{s:measures}, and in particular recalling equation \eqref{form22}, Theorem \ref{mainTechnical} will be proven if we manage to establish that
\begin{equation}\label{form23} \int_{\mathcal{L}_{\angle}} Xf(\ell)^{2} \, d\mathfrak{m}(\ell) \leq \delta^{-\eta}, \end{equation}
assuming that $\epsilon,\delta > 0$ are small enough, depending on $\eta$. Recall that $\mathcal{L}_{\angle} = \ell(\{(a,b,c) : |a| \leq 1\})$. To estimate the quantity in \eqref{form23}, notice first that
\begin{equation}\label{form24} Xf(\ell) = \int_{\ell} f \, d\mathcal{H}^{1} \lesssim (\delta^{3}|\mathcal{B}|)^{-1} \cdot |\{B \in \mathcal{B} : \ell \cap B \neq \emptyset\}|, \qquad \ell \in \mathcal{L}_{\angle}, \end{equation}
because $\mathcal{H}^{1}(B \cap \ell) \lesssim \delta$ for all $B \in \mathcal{B}$. Write $N(\ell) := |\{B \in \mathcal{B} : \ell \cap B \neq \emptyset\}|$. Then, as we just saw,
\begin{align*} \int_{\mathcal{L}_{\angle}} Xf(\ell)^{2} \, d\mathfrak{m}(\ell) & \lesssim (\delta^{3}|\mathcal{B}|)^{-2} \int_{\mathcal{L}_{\angle}} N(\ell)^{2} \, d\mathfrak{m}(\ell)\\
& \leq (\delta^{3}|\mathcal{B}|)^{-2} \int_{B(2)} N(\ell(p))^{2} \, d\mathrm{Leb}(p). \end{align*}
The second inequality is based on (a) the definition of the measure $\mathfrak{m} = \ell_{\sharp}\mathrm{Leb}$, and (b) the observation that if $\ell(p) \in \mathcal{L}_{\angle}$ and $N(\ell(p)) \neq 0$, then $\ell(p) \cap B_{\He}(c) \neq \emptyset$, and this forces $p \in B(2)$ (if $c > 0$ was taken small enough). Finally, by Lemma \ref{l:IncidenceDuality}, we have
\begin{displaymath} N(\ell(p)) \leq |\{B \in \mathcal{B} : p \in \ell^{\ast}(B)\}| = \sum_{B \in \mathcal{B}} \mathbf{1}_{\ell^{\ast}(B)}(p). \end{displaymath}
Indeed, whenever $\ell(p) \cap B \neq \emptyset$ for some $B \in \mathcal{B}$, there exists a point $q \in \ell(p) \cap B$, and then Lemma \ref{l:IncidenceDuality} implies that $p \in \ell^{\ast}(q) \subset \ell^{\ast}(B)$. Therefore, combining \eqref{form23}-\eqref{form24}, it will suffice to show that for $\eta > 0$ fixed, the inequality
\begin{equation}\label{form25} \int_{B(2)} \Big( \sum_{B \in \mathcal{B}} \mathbf{1}_{\ell^{\ast}(B)} \Big)^{2} \leq \delta^{-\eta} \cdot (\delta^{3}|\mathcal{B}|)^{2} \end{equation}
holds assuming that we have picked $\epsilon > 0$ (in the $(\delta,3,\delta^{-\epsilon})$-set hypothesis for $\mathcal{B}$) sufficiently small, depending on $\eta$. We formulate a slightly more general version of this inequality in Proposition \ref{prop5} below, and then explain in the remark afterwards why \eqref{form25} is a consequence. This completes the proof of Theorem \ref{mainTechnical}. \end{proof}

\begin{proposition}\label{prop5} For every $\epsilon > 0$, there exists $\delta_{0} > 0$ such that the following holds for all $\delta \in (0,\delta_{0}]$. Let $\mathcal{B}$ be a family of $\delta$-balls contained in $B_{\He}(1)$ with $\delta$-separated centres, and satisfying the following non-concentration condition for some $\mathbf{C} > 0$:
\begin{equation}\label{form38a} |\{B \in \mathcal{B} : B \subset B_{\He}(p,r)\}| \leq \mathbf{C} \cdot \left(\frac{r}{\delta} \right)^{3}, \qquad p \in \He, \, r \geq \delta.  \end{equation}
Then,
\begin{equation}\label{form39a} \int_{B(2)} \Big( \sum_{B \in \mathcal{B}} \mathbf{1}_{\ell^{\ast}(B)} \Big)^{2} \leq \mathbf{C} \cdot \delta^{3 - \epsilon}|\mathcal{B}|. \end{equation}
\end{proposition}

\begin{remark} Why is \eqref{form25} a consequence of \eqref{form39a}? In \eqref{form25}, we assumed that $\mathcal{B}$ is a $(\delta,3,\delta^{-\epsilon})$-set. This implies
\begin{displaymath} |\{B \in \mathcal{B} : B \subset B_{\He}(p,r)\}| \lesssim \delta^{-\epsilon} \cdot r^{3}|\mathcal{B}|, \qquad p \in \He, \, r \geq \delta. \end{displaymath}
Therefore, \eqref{form38a} is satisfied with constant $\mathbf{C} \sim \delta^{3 - \epsilon}|\mathcal{B}|$. Hence \eqref{form39a} implies \eqref{form25} if we choose $\epsilon < \eta/2$ and then $\delta > 0$ sufficiently small. 

We chose to formulate Proposition \ref{prop5} separately because the "meaning" of \eqref{form39a} is easier to appreciate than that of \eqref{form25}: namely, if all the sets $\ell^{\ast}(B)$ had a disjoint intersection inside $B(1)$, then the left hand side of \eqref{form39a} would be roughly $\delta^{3}|\mathcal{B}|$. Thus, \eqref{form39a} tells us that under the non-concentration condition \eqref{form38a}, the sets $\ell^{\ast}(B)$ are nearly disjoint inside $B(1)$, at least at the level of $L^{2}$-norms. \end{remark}

\begin{proof}[Proof of Proposition \ref{prop5}] By the discussion in Section \ref{s:ballPlate}, the intersections $\ell^{\ast}(B) \cap B(2)$ are essentially $\delta$-plates -- rectangles of dimensions $1 \times \delta \times \delta^{2}$ tangent to $\mathcal{C}$. More precisely, for every $B \in \mathcal{B}$, let $\mathcal{P}_{B} \subset \R^{3}$ be a $C\delta$-plate (as in Definition \ref{def:plate}) with the property
\begin{displaymath} \ell^{\ast}(B) \cap B(2) \subset \mathcal{P}_{B}. \end{displaymath}
This is possible by first applying Proposition \ref{prop1} (which yields a modified $2\delta$-plate containing $\ell^{\ast}(B)$), and then the first inclusion in \eqref{form9}, which shows that the intersection of the modified $2\delta$-plate with $B(2)$ is contained in a $C\delta$-plate $\mathcal{P}_{B}$. Now, we will prove \eqref{form39a} by establishing that
\begin{equation}\label{form26} \int \Big( \sum_{B \in \mathcal{B}} \mathbf{1}_{\mathcal{P}_{B}} \Big)^{2} \leq \mathbf{C} \cdot \delta^{3 - \epsilon}|\mathcal{B}|. \end{equation}
Every plate $\mathcal{P}_{B}$ has a \emph{direction}, denoted $\theta(\mathcal{P}_{B})$: this is the direction of the longest axis of $\mathcal{P}_{B}$, or more formally the real number "$y \in [-1,1]$" associated to the line "$L_{y}$" in Definitions \ref{def:plate}. By enlarging the plates $\mathcal{P}_{B}$ slightly (if necessary), we may assume that their directions lie in the set $\Theta := (\delta \Z) \cap [-1,1]$: this is because if two plates coincide in all other parameters, and differ in direction by $\leq \delta$, both are contained in constant enlargements of the other (this is not hard to check). The reason why we may restrict attention to $[-1,1]$ is that all the plates $\mathcal{P}_{B}$ were associated to the balls $B \subset B_{\He}(1)$, and in fact the $y$-coordinate of the centre of $B$ determines the direction of $\mathcal{P}_{B}$ (see \eqref{form28}).

We next sort the family $\{\mathcal{P}_{B}\}_{B \in \mathcal{B}}$ according to their directions:
\begin{displaymath} \{\mathcal{P}_{B} : B \in \mathcal{B}\} =: \bigcup_{\theta \in \Theta} \mathcal{P}(\theta), \end{displaymath}
where $\mathcal{P}(\theta) := \{\mathcal{P}_{B} : \theta(\mathcal{P}_{B}) = \theta\}$. Thus, for $\theta \in \Theta$ fixed, the plates in $\mathcal{P}(\theta)$ are all translates of each other. Also, the plates in $\mathcal{P}(\theta)$ for a fixed $\theta$ have bounded overlap: this follows from the assumption that the balls in $\mathcal{B}$ have $\delta$-separated centres, and uses Lemma \ref{lemma2} (the plates with a fixed direction correspond precisely to Heisenberg balls whose $y$-coordinates are, all, within "$\delta$" of each other).

Write $R := \delta^{-2}$, thus $\delta = R^{-1/2}$, and recall the truncated cone $\Gamma = \Gamma_{R}$ from \eqref{form10}. Since the plates $\mathcal{P} \in \mathcal{P}(\theta)$ are translates of each other, they all have a common dual rectangle $\mathcal{P}^{\ast}_{\theta}$ of dimensions $\sim R \times R^{1/2} \times 1$. The rectangle $\mathcal{P}^{\ast}_{\theta}$ is centred at $0$, but we may translate it by $\sim R$ in the direction of its longest $R$-side (a light ray depending on $\theta$) so that the translate lies in the $O(1)$-neighbourhood of $\Gamma_{R}$. Committing a serious abuse of notation, we will denote this translated dual rectangle again by "$\theta$", and the collection of all these sets is denoted $\Theta$. This notation coincides with the discussion below \eqref{form10}. There is a $1$-to-$1$ correspondence between the directions $\theta \in \Theta = \delta \Z \cap [-1,1]$ and the rectangles $\theta \in \Theta$ defined just above, so the notational inconsistency should not cause confusion.

We gradually move towards applying the inequality \eqref{GWZ} of Guth, Wang, and Zhang. The next task is to define the functions $f_{\theta}$ and $f = \sum_{\theta \in \Theta} f_{\theta}$. Fix $\theta \in \Theta$, $\mathcal{P} \in \mathcal{P}(\theta)$, and let $\varphi_{\mathcal{P}} \in \mathcal{S}(\R^{3})$ be a non-negative Schwartz function with the properties
\begin{enumerate}
\item $\mathbf{1}_{\mathcal{P}} \leq \varphi_{\mathcal{P}} \lesssim 1$,
\item $\varphi_{\mathcal{P}}$ has rapid decay outside $\mathcal{P}$,
\item $\widehat{\varphi}_{\mathcal{P}} \subset \mathcal{P}^{\ast}_{\theta}$.
\end{enumerate}
Here "rapid decay outside $\mathcal{P}$ has" the usual meaning: if $\lambda \mathcal{P}$ denotes a $\lambda$-times dilated, concentric, version of $\mathcal{P}$, then $\varphi(x) \lesssim_{N} \lambda^{-N}$ for all $x \in \R^{3} \, \setminus \, \lambda \mathcal{P}$ (and for any $N \in \N$). Then, define the function
\begin{displaymath} f_{\theta} := \sum_{\mathcal{P} \in \mathcal{P}(\theta)} e_{\theta} \cdot \varphi_{\mathcal{P}}. \end{displaymath}
Here $e_{\theta}$ is a modulation, depending only on $\theta$, such that
\begin{displaymath} \widehat{e_{\theta} \cdot \varphi_{\mathcal{P}}} \subset \theta. \end{displaymath}
Now the function $f = \sum_{\theta \in \Theta} f_{\theta}$ satisfies all the assumptions of the inequality \eqref{GWZ}, so
\begin{align} \int_{\R^{3}} \Big( \sum_{B \in \mathcal{B}} \mathbf{1}_{\mathcal{P}_{B}} \Big)^{2} & = \int_{\R^{3}} \Big( \sum_{\theta \in \Theta} \sum_{\mathcal{P} \in \mathcal{P}(\theta)} \mathbf{1}_{\mathcal{P}_{B}} \Big)^{2} \notag\\
& \leq \int_{\R^{3}} \Big( \sum_{\theta \in \Theta} \Big| \sum_{\mathcal{P} \in \mathcal{P}(\theta)} e_{\theta} \cdot \varphi_{\mathcal{P}} \Big|^{2} \Big)^{2} \notag\\
&\label{form16} = \int_{\R^{3}} |Sf|^{4} \lesssim \sum_{R^{-1/2} \leq s \leq 1} \sum_{d(\tau) = s} \sum_{U \| U_{\tau}} \mathrm{Leb}(U)^{-1}\|S_{U}f\|_{L^{2}}^{4}. \end{align}
Recall the notation on the right hand side, in particular that $\delta = R^{-1/2} \leq s \leq 1$ only runs over dyadic rationals, and the definition of the "partial" square function $S_{U}f$ from \eqref{form12}. The rectangles $U$ are $\Delta$-plates with $\Delta = (s^{2}R)^{-1/2} = s^{-1}\delta$. In particular, every $U$ is essentially the $\ell^{\ast}$-dual of a Heisenberg $\Delta$-ball: this will allow us to control $\|S_{U}f\|_{L^{2}}$ by applying the non-concentration condition \eqref{form38a} between scales $\delta$ and $1$.

By definition,
\begin{equation}\label{form13} \|S_{U}f\|_{2}^{2} = \int_{U} \sum_{\theta \subset \tau} |f_{\theta}|^{2} = \int_{U} \sum_{\theta \subset \tau} \Big( \sum_{\mathcal{P} \in \mathcal{P}(\theta)} \varphi_{\mathcal{P}} \big)^{2} \lessapprox \int_{U} \sum_{\theta \subset \mathcal{\tau}} \sum_{\mathcal{P} \in \mathcal{P}(\theta)} \varphi_{\mathcal{P}}. \end{equation}
Above, and in the sequel, the notation $A \lessapprox B$ means that for every $\rho > 0$, there exists a constant $C_{\rho} > 0$ such that $A \leq C_{\rho}\delta^{-\rho}B$. In \eqref{form13}, the final "$\lessapprox$" inequality follows easily from the rapid decay of the functions $\varphi_{\mathcal{P}}$, and the bounded overlap of the plates $\mathcal{P} \in \mathcal{P}(\theta)$ for $\theta \in \Theta$ fixed.

For $\theta \subset \tau$, each plate $\mathcal{P} \in \mathcal{P}(\theta)$ is contained in some translate of $10U_{\tau}$ (this was discussed above \eqref{form12}), but this translate may not be $U$. Let $\mathbf{U} \supset U$ be an $(R^{\epsilon}\Delta)$-plate which is concentric with $U$. We then decompose the right hand side of \eqref{form13} as
\begin{equation}\label{form14} \int_{U} \sum_{\theta \subset \mathcal{\tau}} \sum_{\mathcal{P} \in \mathcal{P}(\theta)} \varphi_{\mathcal{P}} \leq \int \sum_{\theta \subset \tau} \mathop{\sum_{\mathcal{P} \in \mathcal{P}(\theta)}}_{\mathcal{P} \subset \mathbf{U}} \varphi_{\mathcal{P}} + \int_{U} \sum_{\theta \subset \tau} \mathop{\sum_{\mathcal{P} \in \mathcal{P}(\theta)}}_{\mathcal{P} \not\subset \mathbf{U}} \varphi_{\mathcal{P}}. \end{equation}
Since each $\mathcal{P} \in \mathcal{P}(\theta)$ is contained in element of the tiling $\mathcal{U}_{\tau}$ (consisting of translates of $U$) every plate $\mathcal{P}(\theta)$ with $\mathcal{P} \not\subset \mathbf{U}$ is far away from $U$: more precisely, $R^{\epsilon/2}\mathcal{P} \cap U = \emptyset$. By the rapid decay of $\varphi_{\mathcal{P}}$ outside $\mathcal{P}$, this implies that $\varphi_{\mathcal{P}} \lesssim_{\epsilon} \delta^{100}$ on $U$, and therefore the second term of \eqref{form14} is bounded by, say, $\lesssim_{\epsilon} \delta^{50}$.

We then focus on the first term of \eqref{form14}, and we first note that
\begin{equation}\label{form15} \int\sum_{\theta \subset \tau} \mathop{\sum_{\mathcal{P} \in \mathcal{P}(\theta)}}_{\mathcal{P} \subset \mathbf{U}} \varphi_{\mathcal{P}} \lesssim \delta^{3} \cdot |\{\mathcal{P} : \mathcal{P} \subset \mathbf{U}\}|, \end{equation}
since $\|\varphi_{\mathcal{P}}\|_{L^{1}} \sim \mathrm{Leb}(\mathcal{P}) \sim \delta^{3}$. So, we need to find out how many $\delta$-plates $\mathcal{P}$ are contained in $\mathbf{U}$. Since $\mathbf{U}$ is an $(R^{\epsilon}\Delta)$-plate, it follows from the second inclusion \eqref{form9}, combined with the second inclusion in Proposition \ref{prop1}, that
\begin{displaymath} \mathbf{U} \subset \ell^{\ast}(B_{\He}(p_{U},C R^{\epsilon}\Delta)) =: \ell^{\ast}(B_{U}). \end{displaymath}
for some $p_{U} \in \He$, and for some absolute constant $C > 0$. On the other hand, the plates $\mathcal{P} = \mathcal{P}_{B}$, $B \in \mathcal{B}$, were initially chosen in such a way that $\ell^{\ast}(B) \cap \{(s,y,z) : |s| \leq 1\} \subset \mathcal{P}_{B}$. Thus, whenever $\mathcal{P}_{B} \subset \mathbf{U}$, we have
\begin{displaymath} \ell^{\ast}(B) \cap \{(s,y,z) : |s| \leq 1\} \subset \mathcal{P}_{B} \subset \mathbf{U} \subset \ell^{\ast}(B_{U}). \end{displaymath}
This implies by Proposition \ref{prop2} that $B \subset B_{U}$, where possibly $B_{U}$ was inflated by another constant factor. Thus,
\begin{displaymath} |\{\mathcal{P} : \mathcal{P} \subset \mathbf{U}\}| \lesssim |\{B \in \mathcal{B} : B \subset B_{U}\}|. \end{displaymath}
Using \eqref{form38a}, this will easily yield useful upper bounds for $|\{\mathcal{P} : \mathcal{P} \subset \mathbf{U}\}|$.

To make this precise, we sort the sets "$U$" appearing in \eqref{form16} according to the "richness"
\begin{equation}\label{form19} \rho(U) := |\{B \in \mathcal{B} : B \subset B_{U}\}| \stackrel{\eqref{form38a}}{\leq} \mathbf{C} \cdot \left( \frac{C R^{\epsilon}\Delta}{\delta} \right)^{3}. \end{equation}
For $s \in [R^{-1/2},1]$ fixed, we choose a (dyadic) value $\rho = \rho_{s}$ such that
\begin{equation}\label{form18} \sum_{d(\tau) = s} \sum_{U \| U_{\tau}} \mathrm{Leb}(U)^{-1}\|S_{U}f\|_{L^{2}}^{4} \lessapprox \sum_{d(\tau) = s} \mathop{\sum_{U \| U_{\tau}}}_{\rho(U) \sim \rho} \mathrm{Leb}(U)^{-1}\|S_{U}f\|_{L^{2}}^{4}. \end{equation}
Here "$\lessapprox$" hides a constant of the form $C\log(1/\delta)$. Let $\mathcal{U}(\rho)$ be the collection of sets "$U$" appearing on the right hand side, and let $\mathcal{B}' \subset \mathcal{B}$ be the subset of the original $\delta$-balls which are contained in some ball $B_{U}$, $U \in \mathcal{U}(\rho)$. Then, evidently,
\begin{equation}\label{form17} |\mathcal{B}'| \lesssim \rho \cdot |\mathcal{U}(\rho)| \lesssim R^{C\epsilon}|\mathcal{B}'|. \end{equation}
The factor "$R^{C\epsilon}$" arises from the fact that while distinct sets "$U$" are the duals of essentially disjoint Heisenberg $\Delta$-balls, the inflated balls $B_{U}$ only have bounded overlap, depending on the inflation factor $R^{\epsilon}$.

Now, for $U \in \mathcal{U}(\rho)$, we may estimate \eqref{form15} as follows:
\begin{displaymath} \|S_{U}f\|_{L^{2}}^{2} \lessapprox_{\epsilon} \int\sum_{\theta \subset \tau} \mathop{\sum_{\mathcal{P} \in \mathcal{P}(\theta)}}_{\mathcal{P} \subset \mathbf{U}} \varphi_{\mathcal{P}} \lesssim \delta^{3} \cdot \rho \lesssim \delta^{3} \cdot R^{C\epsilon} \cdot \frac{|\mathcal{B}'|}{|\mathcal{U}(\rho)|}. \end{displaymath}
(In this estimate, we have omitted the term "$\delta^{50}$" from the second part of \eqref{form14}, because this term will soon turn out to be much smaller than the best bounds for what remains.) Plugging this estimate into \eqref{form18}, and observing that $\mathrm{Leb}(U) = \Delta^{3}$, we obtain
\begin{align*} \sum_{d(\tau) = s} \sum_{U \| U_{\tau}} \mathrm{Leb}(U)^{-1}\|S_{U}f\|_{L^{2}}^{4} & \lessapprox_{\epsilon} |\mathcal{U}(\rho)| \cdot \Delta^{-3} \cdot \Big( \delta^{3} \cdot R^{C\epsilon} \cdot \frac{|\mathcal{B'}|}{|\mathcal{U}(\rho)|} \Big)^{2}\\
& = \Delta^{-3} \cdot \delta^{6} \cdot R^{2C\epsilon} \cdot \frac{|\mathcal{B}'|^{2}}{|\mathcal{U}(\rho)|}\\
& \stackrel{\eqref{form19} \& \eqref{form17}}{\lesssim} \mathbf{C} \cdot R^{3C\epsilon} \cdot \delta^{3}|\mathcal{B}|. \end{align*}
Notably, this estimate is independent of "$\Delta$" and the parameter "$s$", so we may finally deduce from \eqref{form16} that
\begin{displaymath} \int_{\R^{3}} \Big( \sum_{B \in \mathcal{B}} \mathbf{1}_{\mathcal{P}_{B}} \Big)^{2} \lessapprox_{\epsilon} \mathbf{C} \cdot R^{3C\epsilon} \cdot \delta^{3}|\mathcal{B}|. \end{displaymath}
Since $R = \delta^{-2}$ and $\epsilon > 0$ was arbitrary, this implies \eqref{form39a} by renaming variables, and the proof of Proposition \ref{prop5} is complete. \end{proof}

\section{Proof of Theorem \ref{main}}\label{s:mainProof}

We recall the statement:

\begin{thm}\label{mainV2} Let $K \subset \He$ be a Borel set with $\Hd K = t \in [2,3]$. Then, $\dim_{\mathrm{E}} \pi_{e}(K) \geq t - 1$ for $\mathcal{H}^{1}$ almost every $e \in S^{1}$. Consequently, $\Hd \pi_{e}(K) \geq 2t - 3$ for $\mathcal{H}^{1}$ almost every $e \in S^{1}$.
\end{thm}

\begin{proof} The lower bound for $\Hd \pi_{e}(K)$ follows immediately from the lower bound for $\dim_{\mathrm{E}}(K)$, combined with a general inequality between Hausdorff dimensions relative to Euclidean and Heisenberg metrics of subsets of $\W_{e}$, see \cite[Theorem 2.8]{MR3047423}. So, we focus on proving that $\dim_{\mathrm{E}}(K) \geq t - 1$ for $\mathcal{H}^{1}$ almost every $e \in S^{1}$.

The first steps of the proof are standard; similar arguments have
appeared, for example the deduction of \cite[Theorem 2]{MR4148151} from \cite[Theorem
1]{MR4148151}. So we only sketch the first part of the proof, and
provide full details where they are non-standard.
First, we may assume that $K \subset
B_{\He}(1)$, and we may assume, applying Frostman's lemma, that $K
= \spt(\mu)$ for some Borel probability measure $\mu$ satisfying
$\mu(B_{\He}(p,r)) \lesssim r^{t}$ for all $p \in \He$ and $r >
0$.

We make the counter assumption that there exists $s \in (1,t)$ such that
\begin{displaymath} \mathcal{H}^{1}(\{e \in S^{1} : \dim_{\mathrm{E}} \pi_{e}(K) \leq s - 1\}) > 0. \end{displaymath}
By several applications of the pigeonhole principle, this assumption can be applied to find the following objects for any $\epsilon > 0$, and for arbitrarily small $\delta > 0$:
\begin{enumerate}
\item A Borel subset $S' \subset S^{1}$ of length $\mathcal{H}^{1}(S') \geq \delta^{\epsilon/2}$.
\item For every $e \in S'$ a collection of $\leq \delta^{1 - s}$ Euclidean $\delta$-discs $\mathcal{W}_{e}$, contained in $\mathbb{W}_{e}$.
\item If $W_{e} := \cup \mathcal{W}_{e}$ and $e \in S'$, then
\begin{equation}\label{form43} \mu(\pi_{e}^{-1}(W_{e})) \geq \delta^{\epsilon/2}. \end{equation}
\end{enumerate}
We claim that (1)-(3) violate Theorem \ref{mainDiscrete} if $\delta,\epsilon > 0$ are small enough. To this end, we first need to construct a relevant $(\delta,t,\delta^{-\epsilon})$-set of (Heisenberg) $\delta$-balls $\mathcal{B}$ contained in $B_{\He}(1)$. Morally, this collection is a $\delta$-approximation of $K = \spt(\mu)$. More precisely,
we need to decompose $K$ to the following subsets:
\begin{displaymath} K_{\alpha} := \{p \in K : \tfrac{\alpha}{2} \leq \mu(B_{\He}(p,\delta)) \leq \alpha\}, \end{displaymath}
where $\alpha > 0$ runs over dyadic rationals with $\alpha \lesssim \delta^{t}$. By one final application of the pigeonhole principle, and recalling \eqref{form43}, one can find a fixed index $\alpha \in 2^{-\N}$ such that
\begin{equation}\label{form44} \mu(\pi_{e}^{-1}(W_{e}) \cap K_{\alpha}) \geq \delta^{\epsilon} \end{equation}
for all $e \in S \subset S'$, where $\mathcal{H}^{1}(S) \geq \delta^{\epsilon}$. In particular, $\mu(K_{\alpha}) \geq \delta^{\epsilon}$. Then, we let $\mathcal{B}$ be a (Vitali) cover of $K_{\alpha}$ by finitely overlapping Heisenberg $\delta$-balls with $(\delta/5)$-separated centres. Note that $\delta^{\epsilon}\alpha^{-1} \lesssim |\mathcal{B}| \lesssim \alpha^{-1}$. Using the definition of $K_{\alpha}$, and the Frostman condition for $\mu$, it is now easy to check that $\mathcal{B}$ is a $(\delta,t,C\delta^{-\epsilon})$-set of $\delta$-balls, where $C$ is roughly the Frostman constant of $\mu$.

Finally, from \eqref{form44} and $\alpha \lesssim |\mathcal{B}|^{-1}$, we deduce that if $e \in S$, then $\pi_{e}^{-1}(W_{e})$ intersects $\gtrsim \delta^{\epsilon}|\mathcal{B}|$ elements of $\mathcal{B}$, since
\begin{displaymath} \delta^{\epsilon} \leq \mu(\pi_{e}^{-1}(W_{e}) \cap K_{\alpha}) \leq \alpha \cdot |\{B \in \mathcal{B} : \pi_{e}^{-1}(W_{e}) \cap B \neq \emptyset\}|, \qquad e \in S. \end{displaymath}
Write $\mathcal{B}_{e} := \{B \in \mathcal{B} : \pi_{e}^{-1}(W_{e}) \cap B \neq \emptyset\}$, thus $|\mathcal{B}_{e}| \gtrsim \delta^{\epsilon}|\mathcal{B}|$. We now arrive at the point where it is crucial that the elements of $\mathcal{W}_{e}$ are Euclidean $\delta$-discs. Namely, if $B \in \mathcal{B}_{e}$, then $\pi_{e}^{-1}(D) \cap B \neq \emptyset$ for some $D \in \mathcal{W}_{e}$. Then, because $D$ is a Euclidean $\delta$-disc, and the Euclidean diameter of $\pi_{e}(B)$ is $\lesssim \delta$, we may conclude that $\pi_{e}(B) \subset 2D$. This could seriously fail if $D$ were a disc in the metric $d_{\He}$. Now, however, we see that
\begin{displaymath} \pi_{e}(\cup \mathcal{B}_{e}) \subset \cup \{2D : D \in \mathcal{W}_{e}\}, \end{displaymath}
and in particular $\mathrm{Leb}(\pi_{e}(\cup \mathcal{B}_{e})) \lesssim \delta^{2} \cdot |\mathcal{W}_{e}| \leq \delta^{3 - s}$ for all $e \in S$. This violates the conclusion of Theorem \ref{mainDiscrete}, and the proof of Theorem \ref{mainV2} is complete. \end{proof}

%%%%%%%%%%%%%%%%%%%%%%%%%%%%
\appendix

\section{Completing $(\delta,t)$-sets to $(\delta,3)$-sets}

In this section, we use the following notation for the \emph{$\delta$-truncated $s$-dimensional Riesz energy} of a Radon measure $\nu$ on $\He$:
\begin{displaymath} I_{s}^{\delta}(\nu) := \iint \frac{d\nu(x)\,d\nu(y)}{d_{\He,\delta}(x,y)^{s + t}}, \end{displaymath}
where $d_{\He,\delta}(x,y) := \max\{d_{\He}(x,y),\delta\}$. We also recall that $\mu \ast_{\He} \nu$ is the Heisenberg convolution of $\mu$ and $\nu$, that is, the push-forward of $\mu \times \nu$ under the group operation $(p,q) \to p \cdot q$.

\begin{proposition}\label{prop:aug} Let $0 \leq s,t \leq 3$ with $s + t \leq 3$, and let $\delta \in (0,\tfrac{1}{2}]$. Let $\mu$ be a Borel probability measure on $B_{\He}(1)$ with $I^{\delta}_{t}(\mu) \leq \mathbf{C}$.  Then, there exists a set $H \subset B_{\He}(1)$ with $|H| \leq \delta^{-s}$ such that the uniformly distributed (discrete) measure $\eta$ on $H$ satisfies
\begin{displaymath} I_{s + t}^{\delta}(\eta \ast \mu) \leq \mathbf{C}', \end{displaymath}
where $\mathbf{C}' \leq C\log(1/\delta)^{C} \cdot \mathbf{C}$ for some absolute constant $C > 0$.
\end{proposition}

%\begin{proposition}\label{prop:aug} Let $0 \leq s,t \leq 3$ with $s + t \leq 3$, and let $\delta \in (0,\tfrac{1}{2}]$. Let $P \subset B_{\He}(1)$ be a $\delta$-separated $(\delta,t,\mathbf{C})$-set. Then, there exists a set $H \subset B_{\He}(1)$ with $|H| \leq \delta^{-s}$ such that $H \cdot P$ contains a $(\delta,s + t,\mathbf{C}')$-set, where $\mathbf{C}' \leq C\log(1/\delta)^{C} \cdot \mathbf{C}$ for some absolute constant $C > 0$.
%\end{proposition}

\begin{proof} Let $Z := \delta \cdot \Z^{3} \cap B_{\He}(1)$ be a grid of Euclidean $\delta$-separated lattice points in $B_{\He}(1)$. Then $|Z| \sim \delta^{-3}$. Let $H_{\omega} \subset Z$ be a random set, where each point of $Z$ is included independently with probability $\delta^{-s}/(2|Z|)$. In particular, $\mathbb{E}_{\omega}|H_{\omega}| = \delta^{-s}/2$. While we use the symbol "$\omega$" to index the elements in the underlying probability space, no explicit reference to this space will be needed. Let $\eta_{\omega}$ be the random measure
\begin{displaymath} \eta_{\omega} := \delta^{s} \sum_{p \in H_{\omega}} \delta_{p} = \delta^{s} \sum_{p \in Z} \mathbf{1}_{H_{\omega}}(p) \cdot \delta_{p}. \end{displaymath}
We claim that
\begin{equation}\label{form31} \mathbb{E}_{\omega} \left( I^{\delta}_{s + t}(\eta_{\omega} \ast_{\He} \mu) \right) = \iint \E_{\omega} \iint \frac{d\eta_{\omega}(p) d\eta_{\omega}(q)}{d_{\He,\delta}(p \cdot x,q \cdot y)^{s + t}} \, d\mu(x) \, d\mu(y) \leq \mathbf{C}'. \end{equation}
for some $\mathbf{C}' \lessapprox \mathbf{C}$. In this argument, the notation "$\lessapprox$" hides a constant of the form $C\log(1/\delta)^{C}$. The inequality \eqref{form31} will complete the proof of the proposition, because $|H_{\omega}| \leq \delta^{-s}$ with probability $\geq \tfrac{1}{2}$ (for $\delta > 0$ small enough), and therefore, by Chebychev's inequality, $I_{s + t}^{\delta}(\eta_{\omega} \ast_{\He} \mu) \lesssim \mathbf{C}'$ for some "$\omega$" with $|H_{\omega}| \leq \delta^{-s}$.

To prove \eqref{form31}, it clearly suffices to establish that
\begin{equation}\label{form29} \E_{\omega} \iint \frac{d\eta_{\omega}(p) d\eta_{\omega}(q)}{d_{\He,\delta}(p \cdot x,q \cdot y)^{s + t}} \lessapprox \frac{1}{d_{\He,\delta}(x \cdot y)^{t}}, \qquad x,y \in \spt(\mu) \subset B_{\He}(1). \end{equation}
By definition of $\eta_{\omega}$, we have
\begin{align*} \iint \frac{d\eta_{\omega}(p) d\eta_{\omega}(q)}{d_{\He,\delta}(p \cdot x,q \cdot y)^{s + t}}
& = \delta^{2s} \sum_{p,q \in Z} \frac{\mathbf{1}_{H_{\omega}}(p)\mathbf{1}_{H_{\omega}}(q)}{d_{\He,\delta}(p \cdot x,q \cdot y)^{s + t}}\\
& = \delta^{2s} \sum_{p \in Z}
\frac{\mathbf{1}_{H_{\omega}}(p)}{d_{\He,\delta}(x,y)^{s + t}} +
\delta^{2s} \sum_{\substack{p,q\in Z\\p \neq q}}
\frac{\mathbf{1}_{H_{\omega}(p)}\mathbf{1}_{H_{\omega}(q)}}{d_{\He,\delta}(p
\cdot x,q \cdot y)^{s + t}} =: \Sigma_{1}(\omega) +
\Sigma_{2}(\omega). \end{align*} We consider the expectations of
$\Sigma_{1}(\omega)$ and $\Sigma_{2}(\omega)$ separately. The
former one is simple, using that
$\E_{\omega}(\mathbf{1}_{H_{\omega}}(p)) = \tn_{\omega}\{p \in
H_{\omega}\} = \delta^{-s}/(2|Z|) \sim \delta^{3 - s}$:
\begin{displaymath} \E_{\omega} \Sigma_{1}(\omega) \sim \delta^{2s} \sum_{p \in Z} \frac{\delta^{3 - s}}{d_{\He,\delta}(x,y)^{s + t}} = \frac{|Z| \cdot \delta^{3 + s}}{d_{\He,\delta}(x,y)^{s + t}} \lesssim \frac{\delta^{s}}{d_{\He,\delta}(x,y)^{s + t}} \leq \frac{1}{d_{\He,\delta}(x,y)^{t}}, \end{displaymath}
recalling that $|Z| \lesssim \delta^{-3}$. To handle the expectation of $\Sigma_{2}(\omega)$, we note that $\{p \in H_{\omega}\}$ and $\{q \in H_{\omega}\}$ are independent events for $p \neq q$, hence
\begin{displaymath} \E_{\omega}\Sigma_{2}(\omega) \sim \delta^{2s} \sum_{\substack{p,q\in Z\\p \neq q}} \frac{\delta^{6 - 2s}}{d_{\He,\delta}(p \cdot x,q \cdot y)^{s + t}} \sim \delta^{6} \sum_{p \in Z} \sum_{\delta \leq r \leq 1} r^{-s - t} |\{q \in Z : d_{\He,\delta}(p \cdot x,q \cdot y) \sim r\}|, \end{displaymath}
where "$r$" runs over dyadic rationals. Since the product "$\cdot$" is not commutative, in general $d_{\He,\delta}(p \cdot x,q \cdot y) \neq d_{\He,\delta}(p \cdot x \cdot y^{-1},q)$, so the set $\{q \in Z : d_{\He,\delta}(p \cdot x,q \cdot y) \sim r\}$ is \textbf{not} contained in a $\He$-ball of radius $\sim r$ around $p \cdot x \cdot y^{-1}$. This is the key inefficiency in the argument, and causes the restriction $s + t \leq 3$: under this restriction, it actually suffices to note that $\{q \in Z : d_{\He,\delta}(p \cdot x,q \cdot y) \sim r\}$ is contained in a Euclidean $Cr$-ball. To see this, note that if $q \in Z$ satisfies $d_{\He,\delta}(p \cdot x,q \cdot y) \lesssim r$ with $r \geq \delta$, then
\begin{displaymath} q \in B_{\He}(p \cdot x,Cr) \cdot y^{-1}. \end{displaymath}
Here $B_{\He}(p \cdot x,Cr)$ is contained in a Euclidean ball of radius $\lesssim r$ (using $r \leq 1$). The same remains true after the right translation by $y^{-1}$, because $|y| \lesssim 1$ (by assumption), and the right translation $z \mapsto z \cdot y^{-1}$ is Euclidean Lipschitz with constant depending only on $|y|$.

Now, since a Euclidean $r$-ball contains $\lesssim (r/\delta)^{3}$ points of $Z$, we see that
\begin{displaymath} \E_{\omega}\Sigma_{2}(\omega) \lesssim \delta^{3} \sum_{p \in Z} \sum_{\delta \leq r \leq 1} r^{3 - s - t} \lessapprox 1 \leq \frac{1}{d_{\He,\delta}(x,y)^{s + t}}, \end{displaymath}
where in the final inequality we used again that $x,y \in
\mathrm{spt}(\mu) \subset B_{\He}(1)$. This completes the proof of
\eqref{form29}, and therefore the proof of the proposition.
\end{proof}

\bibliographystyle{plain}
\bibliography{references}
\end{document}